\documentclass[11pt,leqno]{article}

\usepackage{amsthm,amsfonts,amssymb,amsmath,oldgerm}
\usepackage{fullpage}
\usepackage{graphicx}
\usepackage{mathrsfs}
\usepackage{xcolor}
\usepackage[colorinlistoftodos]{todonotes} 

\numberwithin{equation}{section}

\renewcommand\d{\partial}
\renewcommand\a{\alpha}
\renewcommand\b{\beta}

\newcommand{\R}{\mathbb R}
\newcommand{\C}{\mathbb C}

\newcommand{\RM}{{\mathbb{R}}}
\newcommand{\CM}{{\mathbb{C}}}
\newcommand{\NM}{{\mathbb{N}}}
\newcommand{\ZM}{{\mathbb{Z}}}
\newcommand{\vt}{\widetilde{v}}
\newcommand{\ri}{\mathrm{i}}
\newcommand{\re}{\mathrm{e}}
\newcommand{\de}{\mathrm{d}}
\newtheorem{theorem}{Theorem}[section]
\newtheorem{proposition}[theorem]{Proposition}
\newtheorem{corollary}[theorem]{Corollary}
\newtheorem{lemma}[theorem]{Lemma}
\newtheorem{remark}[theorem]{Remark}
\theoremstyle{definition}
\newtheorem{definition}[theorem]{Definition}

\allowdisplaybreaks[3]

\title{Nonlinear Modulational Dynamics of Spectrally Stable Lugiato-Lefever Periodic Waves}

\author{Mariana Haragus\thanks{FEMTO-ST Institute, Universit\'e Bourgogne-Franche Comt\'e, 15b avenue des Montboucons, 25030 Besan\c con cedex, France; \texttt{mharagus@univ-fcomte.fr}}
\quad Mathew A. Johnson\thanks{Department of Mathematics, University of Kansas, 1460 Jayhawk Boulevard, Lawrence, KS 66045, USA; \texttt{matjohn@ku.edu}}
\quad Wesley R. Perkins\thanks{Department of Mathematics, Lehigh University, 17 Memorial Drive East, Bethlehem, PA 18015, USA; \texttt{wrp211@lehigh.edu}}
\quad Bj\"orn de Rijk\thanks{Karlsruhe Institute of Technology, Englerstra\ss e 2, 76131 Karlsruhe, Germany;  \texttt{bjoern.de-rijk@kit.edu}}}

\date{\today}

\begin{document}

\maketitle

\begin{abstract}
We consider the nonlinear stability of spectrally stable periodic waves in the Lugiato-Lefever equation (LLE), a damped nonlinear Schr\"odinger equation with forcing that arises in nonlinear optics. So far, nonlinear stability of such solutions has only been established against co-periodic perturbations by exploiting the existence of a spectral gap. In this paper, we consider perturbations which are localized, i.e., integrable on the line.  Such localized perturbations naturally yield the absence of a spectral gap, so we must rely on a substantially different method with origins in the stability analysis of periodic waves in reaction-diffusion systems. The relevant linear estimates have been obtained in recent work by the first three authors through a delicate decomposition of the associated linearized solution operator. Since its most critical part just decays diffusively, the nonlinear iteration can only be closed if one allows for a spatio-temporal phase modulation. 
However, the modulated perturbation satisfies a quasilinear equation yielding an apparent loss of regularity. 
To overcome this obstacle, we incorporate tame estimates on the unmodulated perturbation, which satisfies a semilinear equation in which no derivatives are lost, yet where decay is too slow to close an independent iteration scheme.
We obtain nonlinear stability of periodic steady waves in the LLE against localized perturbations with precisely the same decay rates as predicted by the linear theory.
\end{abstract}


\section{Introduction}

We consider the nonlinear stability and asymptotic behavior of periodic steady waves in the Lugiato-Lefever equation (LLE)
\begin{equation}\label{e:LLE}
\partial_t \psi=-\ri\beta\psi_{xx}-(1+\ri\alpha)\psi+\ri|\psi|^2\psi+F,
\end{equation}
with parameters $\alpha,\beta \in \RM$ and $F > 0$. The unknown $\psi=\psi(x,t)$ in~\eqref{e:LLE} is a complex-valued function depending on the temporal variable $t\in\R$ and the spatial variable $x\in\R$. The LLE was derived in 1987 from Maxwell's equations in~\cite{LL87} as a model to study pattern formation within the optical field in a dissipative and nonlinear cavity filled with a Kerr medium and subjected to a continuous laser pump.  In that context, $\psi(x,t)$ represents the field envelope, $\alpha>0$ is a detuning parameter, $|\beta|=1$ is a dispersion parameter, and $F>0$ represents a normalized pump strength. Note that the case $\beta=1$, corresponding to a defocusing nonlinearity, is referred to as the ``normal" dispersion case while $\beta=-1$, corresponding to a focusing nonlinearity, is referred to as the ``anomalous" dispersion case. More recently, the LLE has become a model for high-frequency combs generated by microresonators in periodic optical waveguides, and as such has become the subject of intense study in the physics literature; see, for example,~\cite{CGTM17} and references therein.

Until recently, however, there have been  relatively few mathematically rigorous studies of the Lugiato-Lefever equation~\eqref{e:LLE}.  The main mathematical questions raised by the physical problem concern the existence, dynamics and stability of both periodic and localized stationary solutions. Obtaining periodic stationary solutions $\psi(x,t)=\phi(x)$ of the LLE~\eqref{e:LLE} boils down to finding periodic solutions of the associated profile equation
\begin{equation}\label{e:profile}
\ri\beta\phi'' = -(1+\ri\alpha)\phi + \ri|\phi|^2\phi+F.
\end{equation}
This has been carried out using a variety of methods, including local bifurcation theory~\cite{DH18_2,DH18_1,Go17,MOT1}, global bifurcation theory~\cite{MR17}, and perturbative arguments~\cite{HSS19}. Clearly, such solutions are smooth as~\eqref{e:profile} corresponds to a spatial dynamical system in $\phi$ with smooth nonlinearity.

In this work, we are interested in the \emph{nonlinear stability} of these periodic steady waves against small perturbations which are \emph{localized}, i.e., are integrable on the line, complementing the linear stability analysis carried out in~\cite{HJP20}. To this end, let $\phi$ be a $T$-periodic stationary solution of the LLE~\eqref{e:LLE} and decompose $\phi=\phi_r+ \ri \phi_i$ into its real and imaginary parts. We capture the local dynamics about $\phi$ by considering the perturbed solution $\psi(x,t)=\phi(x)+\vt(x,t)$ of~\eqref{e:LLE}. Writing the perturbation as $\vt=\vt_r+ \ri \vt_i$, we find that the real functions $\vt_r$ and $\vt_i$ satisfy the system
\begin{equation}
\partial_t\left(\begin{array}{c}\vt_r\\\vt_i\end{array}\right)=\mathcal{A}[\phi]\left(\begin{array}{c}\vt_r\\\vt_i\end{array}\right)+\mathcal{N}(\vt), \label{e:pert00}
\end{equation}
where here $\mathcal{N}(\vt)$ is at least quadratic in $\vt$ and $\mathcal A[\phi]$ is the matrix differential operator
\begin{equation}\label{e:Aphi}
\mathcal A[\phi]=- I+\mathcal{J}\mathcal{L}[\phi],
\end{equation}
with
\[
\mathcal{J}=\left(\begin{array}{cc}0&-1\\1&0\end{array}\right),\qquad
\mathcal{L}[\phi] = \left(\begin{array}{cc} -\b \d_x^2 - \a  + 3\phi_{r}^2 + \phi_{i}^2 & 2\phi_{r}\phi_{i} \\
  2\phi_{r}\phi_{i} & -\b \d_x^2 - \a  + \phi_{r}^2 + 3\phi_{i}^2\end{array}\right).
\]

\subsection{Spectral Stability Assumptions}

Naturally, the local dynamics of~\eqref{e:LLE} about the periodic steady wave $\phi$ are heavily influenced by the spectrum of the linearization $\mathcal A[\phi]$. As we are considering localized perturbations, we consider $\mathcal{A}[\phi]$ as a linear differential operator on the Hilbert space $L^2(\R)$ with dense domain $H^2(\R)$.\footnote{Throughout, we will suppress the co-domain and simply write $L^2(\RM)$ or $H^2(\R)$ instead of $L^2(\R,\C^2)$ or $H^2(\R,\C^2)$, and similarly for all other Lebesgue or Sobolev spaces.} Since $\mathcal{A}[\phi]$ has periodic coefficients, standard Floquet-Bloch theory implies that its spectrum as an operator acting on $L^2(\R)$ is entirely essential and comprised of a countable union of continuous curves which, thanks to the spatial translation invariance of~\eqref{e:LLE}, necessarily touches the imaginary axis at the origin. This spectral feature makes the stability analysis for localized perturbations significantly different and more challenging than for two other types of perturbations which are naturally considered for periodic waves, namely co-periodic perturbations, which are periodic with period equal to the one of the steady wave, and subharmonic perturbations, whose period is an integer multiple of the period of the background wave. The spectrum associated with co-periodic and subharmonic perturbations is discrete and the translational eigenvalue at the origin can be separated from the rest of the spectrum. Yet, in our setting of localized perturbations, the best one can hope for is that the spectrum is confined to the open left half-plane except for a single critical curve touching the origin in a quadratic tangency, which leads to the following definition.

\begin{definition}\label{Def:spec_stab} Let $T > 0$. A smooth $T$-periodic stationary solution $\phi$ of~\eqref{e:LLE} is said to be \emph{diffusively spectrally stable} provided the following conditions hold:
\begin{enumerate}
\item the spectrum of the linear operator $\mathcal{A}[\phi]$ given by~\eqref{e:Aphi} and acting on $L^2(\R)$ satisfies \[\sigma(\mathcal{A}[\phi])\subset\{\lambda\in\CM:\Re(\lambda)<0\}\cup\{0\};\]
\item there exists $\theta>0$ such that for any $\xi\in[-\pi/T,\pi/T)$ the spectrum of the Bloch operator $\mathcal{A}_\xi[\phi]:= \mathcal{M}_\xi^{-1} \mathcal A[\phi] \mathcal{M}_\xi$, acting on $L^2_{\mathrm{per}}(0,T)$, satisfies
  \[\Re\,\sigma(\mathcal{A}_\xi[\phi])\leq-\theta \xi^2,\]
  where here $\mathcal{M}_\xi$ denotes the multiplication operator $\left(\mathcal{M}_\xi f\right)(x) = \re^{\ri \xi x} f(x)$.
\item $\lambda=0$ is a simple eigenvalue of the Bloch operator $\mathcal{A}_0[\phi]$, and the derivative $\phi' \in L^2_{\mathrm{per}}(0,T)$ of the periodic wave is an associated eigenfunction.
\end{enumerate}
\end{definition}

Since the pioneering work of Schneider~\cite{S96,S98_1,S98_2}, the above spectral stability assumption has been standard in the analysis of periodic traveling or steady waves in dissipative systems. It has been shown~\cite{DSSS,JNRZ_13_1,JNRZ_13_2,SSSU} to imply important properties regarding the nonlinear dynamics against localized, or general bounded, perturbations, including long-time dynamics of the associated modulation functions. Moreover, extensions of this to systems with more symmetries (hence more spectral curves passing through the origin) are regularly used; see, for example,~\cite{BJNRZ_13,JNRZ_Invent}.

The existence of diffusively spectrally stable periodic steady waves in the LLE~\eqref{e:LLE} was established in~\cite{DH18_1} using local bifurcation theory. Such waves were found in parameter regimes of anomalous dispersion, which were investigated in the original work of Lugiato and Lefever~\cite{LL87}; see Remark~\ref{rem:specstab} directly below for further details.

\begin{remark}{\label{rem:specstab} \upshape
Let $\beta = -1$ and fix $\alpha < 41/30$ in~\eqref{e:LLE}. Upon setting $F_1^2=(1-\alpha)^2+1$, it was shown in~\cite{DH18_1} that there exists $\mu_0>0$ such that for any $\mu\in(0,\mu_0)$ the LLE~\eqref{e:LLE} has at parameter value $F^2=F_1^2+\mu$ an even periodic and smooth steady solution with Taylor expansion
\[
\phi_\mu(x)=\phi^*+\frac{3(\alpha+\ri(2-\alpha))}{F_1\sqrt{41-30\alpha}}\,\cos\left(\sqrt{2-\alpha}\,x\right)\sqrt{\mu}+\mathcal{O}(\mu),
\]
where $\phi^* \in \R$ satisfies the algebraic equation
\[
(1+\ri\alpha)\phi-\ri\phi|\phi|^2=F_1.
\]
These solutions are $T$-periodic with period $T=2\pi/\sqrt{2-\alpha}$, and are diffusively spectrally stable in the sense of Definition~\ref{Def:spec_stab}. We note that $\alpha_c=41/30$ was already identified as an instability threshold in the original work of Lugiato and Lefever~\cite{LL87}.
}
\end{remark}

\subsection{Main Result}

We state our main result, which establishes nonlinear stability of diffusively spectrally stable periodic steady waves in the LLE~\eqref{e:LLE} against localized perturbations.

\begin{theorem} \label{t:Loc_NonLinStab}
Let $T > 0$ and suppose $\phi$ is a smooth $T$-periodic steady solution of~\eqref{e:LLE} that is diffusively spectrally stable.\footnote{These hypotheses on $\phi$ are made throughout the whole paper.} Then, there exist constants $\varepsilon, M > 0$ such that, whenever $v_0\in L^1(\RM) \cap H^4(\RM)$ satisfies
\[
E_0:=\left\|v_0\right\|_{L^1\cap H^4}<\varepsilon,
\]
there exist functions
\[\vt, \gamma \in C\big([0,\infty),H^4(\R)\big) \cap C^1\big([0,\infty),H^2(\R)\big),\]
with $\vt(0) = v_0$ and $\gamma(0) = 0$ such that $\psi(t)=\phi+\vt(t)$ is the unique global solution of~\eqref{e:LLE} with initial condition $\psi(0)=\phi + v_0$, and the inequalities
\[
\max\left\{\left\|\psi(t) - \phi\right\|_{L^2},\left\|\gamma(t)\right\|_{L^2}\right\}\leq ME_0 (1+t)^{-\frac{1}{4}},
\]
and
\[
\max\left\{\left\|\psi\left(\cdot-\gamma(\cdot,t),t\right)-\phi\right\|_{L^2},\left\|\partial_x \gamma(t)\right\|_{H^3}, \left\|\partial_t \gamma(t)\right\|_{H^2}\right\} \leq ME_0 (1+t)^{-\frac{3}{4}},
\]
hold for all $t\geq 0$.
\end{theorem}

Theorem~\ref{t:Loc_NonLinStab} is the first nonlinear stability result for $T$-periodic steady waves in the LLE~\eqref{e:LLE} against localized perturbations. So far, nonlinear (in)stability of such solutions has only been established against co-periodic perturbations with the aid of standard orbital stability techniques, which exploit the presence of a spectral gap and lead to exponential decay of the perturbed solution to a time-dependent phase modulation of the periodic wave; see~\cite{DH18_2,MOT1,MOT2,SS19}.\footnote{The extension of these works to $NT$-periodic, i.e., subharmonic, perturbations with arbitrary but fixed $N>1$, however, is straightforward.} In contrast, Theorem~\ref{t:Loc_NonLinStab} establishes \emph{algebraic} decay of the perturbed solution to a \emph{spatio}-temporal phase modulation of the underlying wave. That is, if $\psi$ is a solution of~\eqref{e:LLE} which is initially close in $L^1(\RM)\cap H^4(\RM)$ to the periodic steady wave $\phi$, then there exists a phase function $\gamma(x,t)$ such that for large time $\psi$ should behave approximately like
\begin{equation}\label{e:loc_mod_approx}
\psi(x,t)\approx \phi(x) + \gamma(x,t)\phi'(x)\approx \phi(x+\gamma(x,t)),\qquad t\gg 1.
\end{equation}

We note that \emph{linear} stability of periodic steady waves in the LLE against localized perturbations has been established in preliminary work by the first three authors; see~\cite{HJP20} and~\S\ref{sec:linstab}. Comparing the linear and nonlinear results one finds that the algebraic decay rates in Theorem~\ref{t:Loc_NonLinStab} are optimal in the sense that they coincide with the sharp rates obtained in the linear result in~\cite{HJP20}. On the other hand, Theorem~\ref{t:Loc_NonLinStab} has stronger regularity assumptions than the linear result in~\cite{HJP20}. The choice of regularity is an artifact of our method and is motivated in Remarks~\ref{rem:regularity}. While we expect that it is possible to allow for less regular initial data, we emphasize the focus of this paper is not to obtain optimal regularity with respect to localized perturbations, but rather to introduce a working scheme.

The absence of a spectral gap in our case of localized perturbations renders our approach to proving Theorem~\ref{t:Loc_NonLinStab} substantially different from those for co-periodic or subharmonic perturbations. We rely on the methodologies developed by Johnson et al.~for the nonlinear stability analysis of periodic traveling waves in reaction-diffusion systems and systems of viscous conservation laws; see~\cite{JNRZ_13_1,JNRZ_Invent,JZ10,JZ_11_1,JZ_11_2}. The relevant linear estimates have already been obtained in~\cite{HJP20} by decomposing the linear solution operator $\re^{\mathcal{A}[\phi]t}$ in a high-frequency part, exhibiting exponential decay, and a low-frequency part, which decays algebraically and accounts for the critical translational mode. However, the \emph{nonlinear} analyses in the aforementioned works of Johnson et al.~seem to not fully extend to the current setting of the LLE. In particular, standard techniques exploiting total parabolicity of the equation are seemingly not available to compensate for an apparent loss of derivatives experienced in the associated nonlinear iteration scheme.

Let us explain how this loss of regularity arises. Since the low-frequency part of the linear solution operator just decays diffusively, the nonlinear iteration can only be closed if one accommodates its translational behavior by allowing for a spatio-temporal phase modulation $\gamma(x,t)$ leading to the ``modulated perturbation"
\begin{equation*}
v(x,t) = \psi(x-\gamma(x,t),t)-\phi(x).
\end{equation*}
As shown in Section \ref{sec:mod_pert}, this yields a quasilinear perturbation equation of the form
\begin{equation}\label{e:pert0}
\left(\partial_t-\mathcal{A}[\phi]\right)\left(v+\gamma\phi'\right) = N\left(v,\gamma_x v_x,\gamma_{xx} v_x, \gamma_t v_x, \gamma_x v_{xx},\gamma_t,\gamma_x,\gamma_{xx},\gamma_{xxx},\gamma_{xt}\right),
\end{equation}
where $N$ is some nonlinear function. Attempting to control the norm of the right-hand-side of~\eqref{e:pert0} in, for example, $H^1$ naturally requires control over the perturbation $v$ in $H^3$.

 In several previous works this loss of derivatives was compensated by using so-called \emph{nonlinear damping estimates}, which are $L^2$-energy estimates of the form
\begin{align}
\partial_t E(t) \leq -\eta E(t) + C\|v(t)\|_{L^2}^2, \label{e:damping} 
\end{align}
where $\eta,C$ are positive constants and $E(t)$ is an energy controlling the norm $\|\partial_x^k v(t)\|_{L^2}^2$ for some $k \in \mathbb{N}$. Integrating this differential inequality yields
\begin{align*} 
E(t) \leq \re^{-\eta t} E(0) + C\int_0^t \re^{-\eta(t-s)} \|v(s)\|_{L^2}^2 \de s,
\end{align*}
which effectively controls the $H^k$-norm of the modulated perturbation $v$ in terms of its $L^2$-norm and the $H^k$-norm of the initial perturbation $v(0)$, thereby allowing one to regain the lost regularity and potentially close the iteration scheme. Although nonlinear damping estimates can be obtained without much effort if the equation is totally parabolic, see for instance~\cite[Proposition~2.5]{JNRZ_13_1} or~\cite[Proposition~4.5]{JZ_11_1}, their existence in other contexts is not guaranteed and, in general,  their derivation could be tedious and lengthy: see, for instance, the delicate analyses~\cite[Appendix~A]{JZN},~\cite[Section~5]{MS04} and~\cite{RZ16} in the case of hyperbolic-parabolic systems.

Despite the linear damping term $-\psi$ present in~\eqref{e:LLE}, we were unable to establish a nonlinear damping estimate of the form~\eqref{e:damping} for the modulated perturbation in the current setting of the LLE. The main reason is that we do not manage to control the derivatives of $v$ arising in the nonlinearity of~\eqref{e:pert0}: see Remark \ref{rem:damping2} for details.
One can, however, show that the linear damping term is sufficient to establish such a nonlinear damping estimate for the ``unmodulated perturbation"
\[
\widetilde{v}(x,t)=\psi(x,t)-\phi(x)
\]
since, in this case the perturbation $\tilde{v}$ satisfies a semilinear equation.  We refer to Section \ref{sec:unmod_pert} and Appendix \ref{sec:alternative} for more details.

Consequently, in this work we adopt a different approach to address the loss of derivatives in the nonlinear iteration scheme, which circumvents the use of nonlinear damping estimates. 
Specifically, we combine the strategies in the works by Johnson et al.~with a recent method developed by Sandstede \& de Rijk in~\cite{RS18} to establish nonlinear stability of periodic traveling waves in planar reaction-diffusion systems against perturbations which are bounded along a line in $\R^2$ and decay in the distance from this line. Since such non-integrable perturbations prohibit the use of $L^2$-estimates (and thus, in particular, nonlinear damping estimates), the nonlinear analysis in~\cite{RS18} is based on pointwise estimates and their approach to controlling regularity is to incorporate the unmodulated perturbation $\vt(x,t) = 
\psi(x,t) - \phi(x)$ into the nonlinear iteration scheme, which, in our case, satisfies the semilinear equation~\eqref{e:pert00} obtained by setting $\gamma \equiv 0$ in~\eqref{e:pert0}. While the Duhamel's principle based iteration scheme associated to $\vt$ does not experience a loss of derivatives, the associated decay rates of $\vt$ are too slow to close an independent iteration scheme; see Remark~\ref{rem:fail}.  

Nevertheless, our work shows that the proof of Theorem~\ref{t:Loc_NonLinStab} follows by \emph{coupling} the iteration schemes for the modulated and unmodulated perturbations, exploiting a subtle trade-off between smoothing and decay. That is, the algebraically decaying low-frequency part of the evolution semigroup $\re^{\mathcal{A}[\phi]t}$ is infinitely smoothing and thus compensates for a loss of derivatives. In our analysis this is manifested by integration by parts formulas, which move derivatives off factors that may lead to a loss of regularity onto the low-frequency part of the semigroup. On the other hand, since the linearization $\mathcal A[\phi]$ of the weakly dissipative LLE is obviously not a sectorial operator, the high-frequency part of $\re^{\mathcal{A}[\phi]t}$ cannot be infinitely smoothing, yet it exhibits exponential decay, so that tame bounds on the derivatives of the unmodulated perturbation can be used to compensate for the loss of derivatives and close the nonlinear iteration.

\begin{remark} \label{rem:NLdamping}
{\upshape
In the proof of Theorem~\ref{t:Loc_NonLinStab} we obtain tame bounds on the unmodulated perturbation through iterative estimates on its Duhamel 
formulation, coupled with bounds on the modulated perturbation.
However, since equation~\eqref{e:pert00} is semilinear and inherits the linear damping term $-\vt$ from the LLE, it is not suprising that a nonlinear damping estimate of the form~\eqref{e:damping} can be obtained for the unmodulated perturbation. As outlined in Appendix~\ref{sec:alternative}, this nonlinear damping estimate (again, coupled with bounds on the modulated perturbation) yields an alternative to establish tame bounds on the unmodulated perturbation. We emphasize once again that we were unable to extend this nonlinear damping estimate to the \emph{modulated} perturbation equation, cf.~Remark~\ref{rem:damping2}.}
\end{remark}

\subsection{Outline of Paper}

In Section~\ref{sec:prelim} we review several preliminary results, including Floquet-Bloch theory and the characterization of the spectrum of $\mathcal{A}[\phi]$ in terms of the one-parameter family of Bloch operators $\mathcal{A}_\xi[\phi]$. In Section~\ref{sec:linear_est} we collect and extend the relevant linear results obtained in~\cite{HJP20}. That is, we decompose the semigroup $\re^{\mathcal{A}[\phi]t}$ in low- and high-frequency parts, state associated $L^2$- and $H^m$-estimates and establish integration by parts formulas. In Section~\ref{sec:nonl_it}, we detail the construction of our coupled iteration scheme, as well as explain our general strategy for compensating for the resulting loss of derivatives. Finally, Section~\ref{sec:nonlinearstab} is devoted to the proof of Theorem~\ref{t:Loc_NonLinStab}, and we include in Appendix~\ref{app:local} some details of the local existence and regularity theory utilized in our nonlinear analysis.

\subsection{Discussion and Outlook} \label{sec:outlook}

This work establishes the first nonlinear stability result of steady $T$-periodic waves in the LLE~\eqref{e:LLE} against localized perturbations, underlining their robustness, i.e., they are stable against larger classes of perturbations than only the co-periodic or subharmonic ones. Our work indicates that the methodology developed by Johnson et al.~for reaction-diffusion systems and systems of viscous conservation laws can be extended to wide classes of semilinear evolution equations, even to those that do not admit nonlinear damping estimates to control higher-order derivatives. In fact, we expect that our approach works in the semilinear setting as long as the linearization about the periodic wave generates a semigroup, which can be decomposed in low- and high-frequency part, where the diffusive low-frequency part is infinitely smoothing and the high-frequency part decays exponentially. Here, one could extend the class of initial data by looking at modulated initial conditions of the form
\[\psi(x,0) = \phi(x - \gamma_0(x)) + v_0(x),\]
where the phase off-set $\gamma_0$ might be nonlocalized as in~\cite{JNRZ_13_1,JNRZ_13_2}. \

We also expect that our method could substantially improve nonlinear stability results for periodic waves of the LLE (and for semilinear equations more generally) to subharmonic perturbations, i.e., $NT$-periodic perturbations with $N \in \mathbb{N}$. For example, we point out in the case of the LLE that current techniques~\cite{SS19} exploit the presence of a spectral gap yielding stability results, which are not uniform in $N$, since the exponential decay rate and the allowable size of initial perturbations are both controlled by the size of the spectral gap which tends to zero as $N\to\infty$; see~\cite{HJP20}
and also \cite{JP21,JP22}. We aim, through an extension of the stability theory for localized perturbations presented in this paper, to establish a nonlinear stability result against subharmonic perturbations, which is uniform in $N$; see the forthcoming work~\cite{HJPR2}.

\medskip

\noindent
{\bf Acknowledgments:}
MH was partially supported by the EUR EIPHI program (Contract No. ANR-17-EURE-0002) and the ISITE-BFC project (Contract No. ANR-15-IDEX-0003).
MJ was partially supported by the National Science Foundation under grant number DMS-2108749, as well as the Simons Foundation Collaboration Grant number 714021.

\section{Preliminaries} \label{sec:prelim}

We begin by reviewing some elements of Floquet-Bloch theory, and then record spectral and semigroup properties of the linearization $\mathcal A[\phi]$ about the smooth $T$-periodic steady wave solution $\phi$ of the LLE~\eqref{e:LLE} under the diffusive spectral stability assumption. These properties are $H^m$-analogues of the ones obtained for $L^2$-spaces in~\cite{HJP20}.

\subsection{Floquet-Bloch Theory}

Floquet-Bloch theory is a standard tool for the analysis of linear differential operators with periodic coefficients. It relies upon a Bloch decomposition of functions $g\in L^2(\R)$,
\begin{equation}\label{e:Bloch}
g(x)=\frac{1}{2\pi}\int_{-\pi/T}^{\pi/T}\re^{\ri\xi x}\check{g}(\xi,x)\de\xi,\qquad{\rm where } \ \check{g}(\xi,x):=\sum_{\ell\in\ZM}\re^{2\pi \ri\ell x/T}\hat{g}(\xi+2\pi \ell/T),
\end{equation}
and $\hat{g}(\cdot)$ denotes the Fourier transform of $g$,
\[
\hat{g}(\xi)= \int_{-\infty}^\infty \re^{-\ri\xi x}g(x)\de x.
\]
The equality~\eqref{e:Bloch} is a consequence of the inverse Fourier transform formula,
\[
g(x)=\frac1{2\pi}\int_{-\infty}^\infty \re^{\ri\xi x}\hat{g}(\xi)\de\xi
=\frac1{2\pi}\sum_{\ell\in\ZM}\int_{-\pi/T}^{\pi/T} \re^{\ri(\xi+2\pi \ell/T)x}
\hat{g}(\xi+2\pi \ell/T)\de\xi
=\frac1{2\pi}\int_{-\pi/T}^{\pi/T}\re^{\ri\xi x}\check{g}(\xi,x)\de\xi.
\]
Thus, Parseval's equality implies
\begin{eqnarray}\label{e:parseval_0}
\|g\|^2_{L^2(\RM)}
&=& \frac1{2\pi}\int_{-\infty}^\infty|\hat g(\xi)|^2\de\xi =
\frac1{2\pi}\sum_{\ell\in\ZM}\int_{-\pi/T}^{\pi/T} |\hat{g}(\xi+2\pi \ell/T)|^2\de\xi
\\ \nonumber
&=&\frac{1}{2\pi T} \int_{-\pi/T}^{\pi/T}\int_0^T\left|\check{g}(\xi,x)\right|^2\de x\,\de\xi
=\frac{1}{2\pi T} \| \check{g} \|^2_{L^2\left([-\pi/T,\pi/T);L^2_{\rm per}(0,T)\right)}.
\end{eqnarray}
In particular, the Bloch transform
\[
\mathcal{B}\colon L^2(\RM)\to L^2\left([-\pi/T,\pi/T);L^2_{\rm per}(0,T)\right),
  \qquad \mathcal Bg=\check g,
  \]
is a bounded linear operator.
For fixed $m \in \mathbb N$, we have an analogue of Parseval's equality for the $H^m$-norm,\footnote{Throughout the paper, the notation $A\lesssim B$ means that there exists a constant $C>0$, independent of $A$ and $B$, such that $A\leq CB$, and we write $A\simeq B$ if $A\lesssim B$ and $B\lesssim A$.}
\begin{eqnarray}\label{e:parseval_m}
\|g\|^2_{H^m(\RM)}
&\simeq & \int_{-\infty}^\infty(1+\xi^2)^m|\hat g(\xi)|^2\de\xi \simeq
\sum_{\ell\in\ZM}\int_{-\pi/T}^{\pi/T}(1+(2\pi\ell/T)^2)^m |\hat{g}(\xi+2\pi \ell/T)|^2\de\xi
\\ \nonumber
&\simeq& \int_{-\pi/T}^{\pi/T}\|\check{g}(\xi,\cdot)\|^2_{H^m_{\rm per}(0,T)}\de\xi
 = \| \check{g} \|^2_{L^2\left([-\pi/T,\pi/T);H^m_{\rm per}(0,T)\right)},
\end{eqnarray}
with $g \in H^m(\R)$, yielding that the Bloch transform can also be regarded as a bounded linear operator
\[
\mathcal{B}\colon H^m(\RM)\to L^2\left([-\pi/T,\pi/T);H^m_{\rm per}(0,T)\right).
  \]
In particular, for each $g \in H^1(\R)$ we have that $\mathcal{B}(g)(\xi,\cdot)$ is differentiable with
\begin{equation}\label{e:partialx}
\partial_x\mathcal{B}(g)(\xi,x)=\mathcal{B}\left(\partial_x g\right)(\xi,x)-\ri\xi\check{g}(\xi,x).
\end{equation}

Taking a differential operator $\mathcal A$ with smooth $T$-periodic coefficients acting on $L^2(\R)$, the associated Bloch operators are defined by
\[
\mathcal A_\xi=\mathcal{M}_\xi^{-1} \mathcal{A}\mathcal{M}_\xi,\qquad \xi\in[-\pi/T,\pi/T),
\]
where here $\mathcal{M}_\xi$ denotes the multiplication operator $\left(\mathcal{M}_\xi f\right)(x) = \re^{\ri \xi x} f(x)$. The operators $\mathcal A_\xi$ act in $L^2_{\rm per}(0,T)$, and their dependency on $\xi$ is analytic. For $v\in D(\mathcal A)$ we have the representation formula
\[
\mathcal{A}v(x)=\frac{1}{2\pi}\int_{-\pi/T}^{\pi/T} \re^{\ri\xi x}\mathcal{A}_\xi\check{v}(\xi,x)\de\xi,
\]
and a similar formula holds for the associated semigroups, provided they exist,
\begin{equation}\label{e:blochsoln}
\re^{\mathcal{A}t}v(x)=\frac{1}{2\pi}\int_{-\pi/T}^{\pi/T} \re^{\ri\xi x}\re^{\mathcal{A}_\xi t}\check{v}(\xi,x)\de\xi.
\end{equation}
An important property of the Bloch operators $\mathcal A_\xi$ is that their domains are compactly embedded in $L^2_{\rm per}(0,T)$, and therefore their spectra consist entirely of isolated eigenvalues of finite algebraic multiplicities. The spectral decomposition formula
\[
  \sigma\left(\mathcal{A}\right)=\bigcup_{\xi\in[-\pi/T,\pi/T)}\sigma\left(\mathcal{A}_\xi\right),
\]
characterizes the $L^2(\RM)$-spectrum of $\mathcal{A}$ as the union of countably many continuous curves $\lambda(\xi)$ corresponding to the eigenvalues of the associated Bloch operators $\mathcal{A}_\xi$.
We refer to~\cite[Section 2]{HJP20} for more details and further properties.

\subsection{Spectral Properties}

The Bloch operators associated with the periodic differential operator $\mathcal A[\phi]$ given by~\eqref{e:Aphi} are defined for $\xi\in[-\pi/T,\pi/T)$ by the formula
\begin{equation*}
\mathcal A_\xi[\phi]=- I+\mathcal{J}\mathcal{L}_\xi[\phi],
\end{equation*}
where
\[
\mathcal{J}=\left(\begin{array}{cc}0&-1\\1&0\end{array}\right),\qquad
\mathcal{L}_\xi[\phi] = \left(\begin{matrix} -\b (\partial_x+\ri\xi)^2 - \a  + 3\phi_{r}^2 + \phi_{i}^2 & 2\phi_{r}\phi_{i} \\
  2\phi_{r}\phi_{i} & -\b (\partial_x+\ri\xi)^2 - \a  + \phi_{r}^2 + 3\phi_{i}^2\end{matrix}\right).
  \]
As the $T$-periodic solution $\phi$ of the LLE~\eqref{e:LLE} is smooth, the operators $\mathcal A_\xi[\phi]$ are closed in $L^2_{\rm per}(0,T)$ and $H^m_{\rm per}(0,T)$, for any $m\in\NM$, with compactly embedded domains $H^2_{\rm per}(0,T)$ and $H^{m+2}_{\rm per}(0,T)$, respectively. A standard bootstrapping argument, just as for stationary solutions of~\eqref{e:LLE}, shows that eigenfunctions and generalized eigenfunctions of $\mathcal A_\xi[\phi]$ are smooth. As a consequence, the operators $\mathcal A_\xi[\phi]$ have the same spectral properties when acting on $L^2_{\rm per}(0,T)$ or on $H^m_{\rm per}(0,T)$, and the following lemma which is a direct consequence of the diffusive spectral stability of $\phi$ and which was proved in~\cite{HJP20} for $L^2_{\rm per}(0,T)$, remains valid for $H^m_{\rm per}(0,T)$.

\begin{lemma}[Spectral Preparation]\label{L:spec_prep}
The Bloch operators $\mathcal A_\xi[\phi]$ acting on $H^m_{\rm per}(0,T)$, for some
$m\in\NM_0$, have the following properties.\footnote{We use the notation $H^0_{\rm per}(0,T)=L^2_{\rm per}(0,T)$.}
  \begin{enumerate}
\item For any fixed $\xi_0\in(0,\pi/T)$, there exists a positive constant $\delta_0$  such that
\[
\Re\,\sigma(\mathcal{A}_\xi[\phi])<-\delta_0,
\]
for all $\xi\in[-\pi/T,\pi/T)$ with $|\xi|>\xi_0$.
  \item There exist constants $\xi_1 \in (0,{\pi}/{T})$ and $\delta_1 > 0$ such that for any $|\xi|<\xi_1$ the spectrum of $\mathcal{A}_\xi[\phi]$ decomposes into two disjoint subsets
\[
\sigma(\mathcal{A}_\xi[\phi])=\sigma_-(\mathcal{A}_\xi[\phi])\cup\sigma_0(\mathcal{A}_\xi[\phi]),
\]
with the following properties:
\begin{enumerate}
\item $\Re\,\sigma_-(\mathcal{A}_\xi[\phi])<-\delta_1$ and $\Re\,\sigma_0(\mathcal{A}_\xi[\phi])>-\delta_1$;
\item the set $\sigma_0(\mathcal{A}_\xi[\phi])$ consists of a single  eigenvalue $\lambda_c(\xi)$ which is simple, analytic in $\xi$, and expands as
\begin{equation}\label{e:lambda_c}
\lambda_c(\xi)=\ri a\xi- d \xi^2+\mathcal{O}(|\xi|^3),
\end{equation}
for some $a\in\RM$ and $d>0$;
\item the eigenfunction $\Phi_\xi$ associated with $\lambda_c(\xi)$ is a smooth function, depends analytically on $\xi$, and there exists a constant $C > 0$ such that
\[
\left\|\Phi_\xi-\phi'\right\|_{H^m_{\rm per}(0,T)} \leq C |\xi|,
\]
where $\phi'$ is the derivative of the $T$-periodic solution $\phi$.
\end{enumerate}
\end{enumerate}
\end{lemma}

We point out that the expansion~\eqref{e:lambda_c} of the simple eigenvalue $\lambda_c(\xi)$ is a consequence of the property
\[
\overline{\mathcal A_\xi[\phi]} = \mathcal A_{-\xi}[\phi],
\]
which holds in general for Bloch operators $\mathcal A_\xi$ associated with real periodic differential operators $\mathcal A$. So, the eigenvalue $\lambda_c(\xi)$, being the only eigenvalue of $\mathcal A_\xi[\phi]$ with real part larger than $-\delta_1$ for all $|\xi|<\xi_1$, satisfies $\overline{\lambda_c(\xi)}=\lambda_c(-\xi)$, which gives the expansion~\eqref{e:lambda_c}. In addition, if the periodic solution $\phi$ of the LLE is an even function, which is the case for the diffusively spectrally stable periodic solutions constructed in~\cite{DH18_1}, then the operator $\mathcal A[\phi]$ is invariant under the reflection $x\mapsto-x$, and the Bloch operators satisfy
\[
R\mathcal A_\xi[\phi] = \mathcal A_{-\xi}[\phi]R, \qquad (Rv)(x)=v(-x),
\]
which implies $\lambda_c(\xi)=\lambda_c(-\xi)$ and gives $a=0$ in the expansion~\eqref{e:lambda_c} in this case.

Finally, notice that the adjoint operator $\mathcal{A}^*_\xi[\phi]$ has similar spectral properties, its spectrum being equal to the complex conjugated spectrum of $\mathcal A_\xi[\phi]$. In particular, $\overline{\lambda_c(\xi)}$ is a simple eigenvalue of $\mathcal{A}^*_\xi[\phi]$ with smooth associated eigenfunction $\widetilde{\Phi}_\xi$ depending analytically on $\xi$.

\subsection{Semigroup Properties}

As for the spectral properties above, the semigroup properties of the operators  $\mathcal A_\xi[\phi]$ are the same when acting on $L^2_{\rm per}(0,T)$ or on $H^m_{\rm per}(0,T)$, for any $m\in\NM$. The following result proved in~\cite{HJP20} for $L^2_{\rm per}(0,T)$ remains valid in $H^m_{\rm per}(0,T)$.

\begin{lemma}[Bloch semigroups]\label{L:Bloch_sg}
The Bloch operators $\mathcal A_\xi[\phi]$ acting on $H^m_{\rm per}(0,T)$, for some $m\in\NM_0$, generate $C^0$-semigroups with the following properties.
\begin{enumerate}
\item For any fixed $\xi_0\in(0,\pi/T)$, there exist positive constants $C_0$ and $\mu_0$  such that
\[
\left\|\re^{\mathcal{A}_\xi[\phi] t}\right\|_{\mathcal L(H^m_{\rm per}(0,T))}\leq C_0\re^{-\mu_0t},
\]
for all $t\geq0$ and all $\xi\in[-\pi/T,\pi/T)$ with $|\xi|>\xi_0$.
  \item With $\xi_1$ chosen as in  Lemma~\ref{L:spec_prep}~(ii), there exist positive constants  $C_1$ and $\mu_1$ such that for any $|\xi|<\xi_1$, if $\Pi(\xi)$ is the spectral projection onto the (one-dimensional) eigenspace associated with the eigenvalue  $\lambda_c(\xi)$  given by Lemma~\ref{L:spec_prep}~(ii), then
\[
\left\|\re^{\mathcal{A}_\xi[\phi]t}\left(I-\Pi(\xi)\right)\right\|_{\mathcal L(H^m_{\rm per}(0,T))}\leq C_1 \re^{-\mu_1t},
\]
for all $t\geq0$.
  \end{enumerate}
\end{lemma}

\section{Linear Estimates} \label{sec:linear_est}

In this section, we review the decomposition of the evolution semigroup $\re^{\mathcal{A}[\phi]t}$ in a low- and high-frequency part that was recently obtained in the work~\cite{HJP20}. We extend the $L^2$-estimates from~\cite{HJP20} on the high-frequency part to $H^m$-estimates, which will be needed in our subsequent nonlinear stability analysis. Moreover, we exploit the smoothing properties of the low-frequency part to extend the $L^2 \cap L^1 \to L^2$-estimates from~\cite{HJP20} to $L^2 \cap L^1 \to H^m$-estimates and establish associated integration by part identities.

\subsection{Decomposition of the Evolution Semigroup}

Following~\cite{HJP20}, we decompose the $C^0$-semigroup $\re^{\mathcal{A}[\phi]t}$ in an exponentially decaying part and a critical part exhibiting algebraic decay. Take $\xi_1\in(0,{\pi}/{T})$  as in Lemma~\ref{L:Bloch_sg} and a smooth nonnegative cut-off function $\rho$ satisfying $\rho(\xi)=1$ for $|\xi|<{\xi_1}/{2}$ and $\rho(\xi)=0$ for $|\xi|>\xi_1$. For each $|\xi|<\xi_1$, consider the spectral projection $\Pi(\xi)$ onto the one-dimensional eigenspace of $\mathcal{A}_\xi[\phi]$ associated with the eigenvalue $\lambda_c(\xi)$, given explicitly by
\[
\Pi(\xi)g = \left\langle\widetilde\Phi_\xi,g\right\rangle_{L^2(0,T)}\Phi_\xi,
\]
for any $g\in L^2_{\mathrm{per}}(0,T)$, where $\widetilde\Phi_\xi$ is the smooth eigenfunction of the adjoint operator $\mathcal{A}^*_\xi[\phi]$ associated with the eigenvalue $\overline{\lambda_c(\xi)}$ that satisfies \smash{$\langle\widetilde\Phi_\xi,\Phi_\xi\rangle_{L^2(0,T)}=1$.}

Starting from the representation formula~\eqref{e:blochsoln}, we write
\begin{align}
\re^{\mathcal{A}[\phi]t}v(x)=\frac{1}{2\pi}\int_{-\pi/T}^{\pi/T}\re^{\ri\xi x}\re^{\mathcal{A}_\xi[\phi] t}\check{v}(\xi,x)\de \xi=S_c(t)v(x) + S_e(t)v(x), \label{e:semidecomp1}
\end{align}
for $t \geq 0$ and $x \in \R$, with
\[
S_c(t)v(x):= \frac{1}{2\pi}\int_{-\pi/T}^{\pi/T}\rho(\xi)\re^{\ri\xi x}\re^{\mathcal{A}_\xi[\phi] t}\Pi(\xi)\check{v}(\xi,x)\de \xi,
\]
and
\[
S_e(t)v(x):=\frac{1}{2\pi}\int_{-\pi/T}^{\pi/T} \left(1-\rho(\xi)\right)\re^{\ri\xi x}\re^{\mathcal{A}_\xi[\phi]t}\check{v}(\xi,x)\de \xi
			+\frac{1}{2\pi}\int_{-\pi/T}^{\pi/T}\rho(\xi)\re^{\ri\xi x}\re^{\mathcal{A}_\xi[\phi]t}\left(1-\Pi(\xi)\right)\check{v}(\xi,x)\de \xi.
\]
The component $S_e(t)$ is the exponentially decaying part of the evolution semigroup $\re^{\mathcal{A}[\phi]t}$.

\begin{lemma}[Exponential decay]\label{L:unmod_bd} For any integer $m\geq0$, there exist constants $\mu, C>0$ such that the inequality
\[
\left\|S_e(t)v\right\|_{\mathcal L(H^m)} \leq C \re^{-\mu t},
\]
holds for any $t\geq 0$.
\end{lemma}
\begin{proof}
This estimate has been established in the case $m=0$ in~\cite{HJP20}. The proof relies upon Parseval's equality for $L^2$-functions~\eqref{e:parseval_0} and the $L^2$-estimates on the Bloch semigroups from Lemma~\ref{L:Bloch_sg}. It is easily transferred to $m\in\NM$ using Parseval's equality for $H^m$-functions~\eqref{e:parseval_m} and the $H^m$-estimates for the Bloch semigroups in Lemma~\ref{L:Bloch_sg}.
\end{proof}

We continue by further decomposing the critical component $S_c(t)$ of the semigroup in order to identify its slowest decaying component. We introduce a smooth cut-off function $\chi\colon[0,\infty) \to \R$, which vanishes on $[0,1]$ and equals $1$ on $[2,\infty)$.\footnote{ The reason for introducing the cut-off function $\chi(t)$ becomes apparent only in the forthcoming nonlinear stability analysis; we refer to Remark~\ref{rem:cut-off} for further details.} Using the explicit formula for the spectral projection $\Pi(\xi)$ and Lemma~\ref{L:spec_prep} we write
\begin{align}
\begin{split}
S_c(t)v(x) &= \chi(t) S_c(t) v(x) + (1-\chi(t))S_c(t) v(x)\\
&=\frac{\chi(t)}{2\pi}\int_{-\pi/T}^{\pi/T}\rho(\xi)\re^{\ri\xi x+\lambda_c(\xi)t}\left<\widetilde{\Phi}_\xi,\check{v}(\xi,\cdot)\right>_{L^2(0,T)}\Phi_\xi(x) \de \xi + (1-\chi(t))S_c(t) v(x)\\
&=\phi'(x)\left(\frac{\chi(t)}{2\pi}\int_{-\pi/T}^{\pi/T}\rho(\xi)\re^{\ri\xi x+\lambda_c(\xi)t}\left<\widetilde{\Phi}_\xi,\check{v}(\xi,\cdot)\right>_{L^2(0,T)} \de \xi\right) + (1-\chi(t))S_c(t) v(x)\\
&\quad+\frac{\chi(t)}{2\pi}\int_{-\pi/T}^{\pi/T}\rho(\xi)\re^{\ri\xi x+\lambda_c(\xi)t}\ri\xi\left(\frac{\Phi_\xi(x)-\phi'(x)}{\ri\xi}\right)\left<\widetilde{\Phi}_\xi,\check{v}(\xi,\cdot)\right>_{L^2(0,T)} \de \xi \\
&=: \phi'(x)s_p(t)v(x) + \widetilde{S}_c(t)v(x),
\end{split} \label{e:semidecomp2}
\end{align}
with
\begin{align*}
s_p(t)v(x) = \frac{\chi(t)}{2\pi}\int_{-\pi/T}^{\pi/T}\rho(\xi)\re^{\ri\xi x+\lambda_c(\xi)t}\left<\widetilde{\Phi}_\xi,\check{v}(\xi,\cdot)\right>_{L^2(0,T)} \de \xi,
\end{align*}
for $t \geq 0$ and $x \in \R$. The motivation behind the decomposition~\eqref{e:semidecomp2} is that the $s_p(t)$-contribution precisely captures the slowest (diffusive) decay at rate $(1+t)^{-1/4}$ exhibited by $S_c(t)$, whereas the remaining part $\widetilde{S}_c(t)$ decays faster at rate $(1+t)^{-3/4}$. The following lemma establishes these algebraic decay properties, which are needed in the upcoming nonlinear analysis.

\begin{lemma}[Critical Component]\label{L:mod_bd}
For all integers $\ell,j,m\geq 0$ there exist constants $C_{\ell,j},C_m>0$ such that
\begin{align*}
 &\left\|\partial_x^\ell\partial_t^j s_p(t)v\right\|_{L^2} \leq C_{\ell,j}(1+t)^{-\frac{\ell+j}{2}}\|v\|_{L^2}, \qquad v\in L^2(\R),\\
 &\left\|\partial_x^\ell\partial_t^j s_p(t)v\right\|_{L^2} \leq C_{\ell,j}(1+t)^{-\frac{1}{4}-\frac{\ell+j}{2}}\|v\|_{L^1} ,\qquad v\in L^1(\R)\cap L^2(\R),
\end{align*}
and
\begin{align*}
  &\left\|\partial_x^m\widetilde{S}_c(t)v\right\|_{L^2}\leq C_m(1+t)^{-\frac{3}{4}}\|v\|_{L^1}, \qquad v\in L^1(\R)\cap L^2(\R),
\end{align*}
for all $t\geq 0$.
\end{lemma}

\begin{remark} \upshape
Lemma~\ref{L:mod_bd} shows that the critical part $S_c(t)$ of the evolution semigroup $\re^{\mathcal A[\phi] t}$ is infinitely smoothing, i.e., it defines a bounded linear map from $L^1(\R)\cap L^2(\R)$ into $H^m(\R)$ for each $m \in \NM_0$. Since the linearization $\mathcal A[\phi]$ of the weakly dissipative LLE is obviously not a sectorial operator, the same cannot be expected for the high-frequency part $S_e(t)$ of $\re^{\mathcal{A}[\phi]t}$.
\end{remark}

\begin{proof} First, from~\cite[Section 3]{HJP20} we have the estimates
  \begin{equation}\label{e:sp}
  \left|\left<\widetilde{\Phi}_\xi,\check{v}(\xi,\cdot)\right>_{L^2(0,T)}\right|
  \lesssim \left\|\check{v}(\xi,\cdot) \right\|_{L^2(0,T)},\qquad
  \left|\left<\widetilde{\Phi}_\xi,\check{w}(\xi,\cdot)\right>_{L^2(0,T)}\right|
  \lesssim \|w\|_{L^1},
  \end{equation}
for $v \in L^2(\R)$ and $w \in L^1(\R) \cap L^2(\R)$. Next, take $t\geq2$, so that $\chi(t)=1$. Then
\[
\partial_x^\ell\partial_t^j s_p(t)v(x) = \frac{1}{2\pi}\int_{-\pi/T}^{\pi/T}\re^{\ri\xi x} \rho(\xi)(\ri\xi)^\ell\left(\lambda_c(\xi)\right)^j\re^{\lambda_c(\xi)t}\left<\widetilde{\Phi}_\xi,\check{v}(\xi,\cdot)\right>_{L^2(0,T)}\de \xi,
\]
and Parseval's equality~\eqref{e:parseval_0} implies that
\[
\left\|\partial_x^\ell\partial_t^j s_p(t)v\right\|_{L^2}^2 =
\frac1{2\pi T}\int_{-\pi/T}^{\pi/T}\int_0^T \left|\rho(\xi)(\ri\xi)^\ell\left(\lambda_c(\xi)\right)^j\re^{\lambda_c(\xi)t}\left<\widetilde{\Phi}_\xi,\check{v}(\xi,\cdot)\right>_{L^2(0,T)}\right|^2\de x\,\de \xi,
\]
for $v \in L^2(\R)$. Using ~\eqref{e:lambda_c} and~\eqref{e:sp} we find
\begin{align*}
  &\left\|\partial_x^\ell\partial_t^j s_p(t)v\right\|_{L^2}
  \lesssim \left\|\xi^{\ell+j}\re^{-\de \xi^2 t}\right\|_{L^\infty_\xi[-\frac{\pi}{T},\frac{\pi}{T})}\|v\|_{L^2}
  \lesssim (1+t)^{-\frac{\ell+j}{2}}\|v\|_{L^2}, \qquad v \in L^2(\R),\\
  &\left\|\partial_x^\ell\partial_t^j s_p(t)w\right\|_{L^2}
  \lesssim \left\|\xi^{\ell+j}\re^{-\de \xi^2 t}\right\|_{L^2_\xi[-\frac{\pi}{T},\frac{\pi}{T})}\|v\|_{L^1}
  \lesssim (1+t)^{-\frac{1}{4}-\frac{\ell+j}{2}}\|w\|_{L^1}, \qquad w \in L^1(\R) \cap L^2(\R),
\end{align*}
which prove the first two inequalities for $t\geq2$. Similarly, for $\widetilde{S}_c(t)$, using in addition that the quantity
\[
\sup_{\xi\in[-\pi/T,\pi/T)} \left\|\partial_x^k\left(\frac{\Phi_\xi-\phi'}{\ri\xi}\right)\right\|_{L^\infty}, \qquad k = 0,\ldots,m,
\]
is finite by Lemma~\ref{L:spec_prep}, we find
\[
\left\|\partial_x^m\widetilde{S}_c(t)v\right\|_{L^2}
\lesssim \left(\left\|\xi^{m+1} \re^{-\de \xi^2 t}\right\|_{L^2_\xi[-\frac{\pi}{T},\frac{\pi}{T})}+\dots+\left\|\xi \re^{-\de \xi^2 t}\right\|_{L^2_\xi[-\frac{\pi}{T},\frac{\pi}{T})}\right)\|v\|_{L^1}
\lesssim (1+t)^{-3/4}\|v\|_{L^1},
\]
which proves the third inequality for $t\geq2$.

The corresponding short-time bounds for $t \in [0,2]$ on  $s_p(t)$ and $\widetilde{S}_c(t)$ follow similarly using that $\chi(t)$, and thus $s_p(t)$, vanishes on $[0,1]$,
and that $\chi(t)$ and its derivatives are bounded on $[1,2]$.
\end{proof}

\subsection{Integration by Parts Identities}\label{sec:parts}

In addition to the above linear estimates, our forthcoming nonlinear iteration scheme requires the following integration-by-parts type identities to move derivatives off factors that may lead to a loss of regularity onto the smoothing low-frequency part of the semigroup.

\begin{proposition}[Integration by Parts] \label{P:parts}
Given $f,g\in H^1(\RM)$, we have the following identities for all $t\geq 0$ and $x \in \R$:
\begin{align*}
s_p(t)\left(f\cdot \partial_x g\right)(x) &= -s_p(t)(\partial_xf\cdot g)+\partial_xs_p(t)(fg)(x)\\
&\quad\qquad-	\frac{\chi(t)}{2\pi}\int_{-\pi/T}^{\pi/T}\rho(\xi)\re^{\ri\xi x+\lambda_c(\xi)t}\left<\partial_x\widetilde\Phi_\xi,\mathcal{B}(fg)(\xi,\cdot)\right>_{L^2(0,T)}\de \xi,
\end{align*}
and
\begin{align*}
\widetilde{S}_c(t)&\left(f\cdot\partial_x g\right)(x)=-\widetilde{S}_c(t)\left(\partial_xf\cdot g\right)\\
&\qquad+\frac{(1-\chi(t))}{2\pi}\int_{-\pi/T}^{\pi/T}\rho(\xi)\re^{\ri\xi x+\lambda_c(\xi)t}\ri\xi\Phi_\xi(x)\left<\widetilde{\Phi}_\xi,\mathcal{B}(fg)(\xi,\cdot)\right>_{L^2(0,T)}\de \xi\\
&\qquad\qquad
-\frac{(1-\chi(t))}{2\pi}\int_{-\pi/T}^{\pi/T}\rho(\xi)\re^{\ri\xi x+\lambda_c(\xi)t}\Phi_\xi(x)\left<\partial_x\widetilde{\Phi}_\xi,\mathcal{B}(fg)(\xi,\cdot)\right>_{L^2(0,T)}\de \xi\\
&\qquad+\frac{\chi(t)}{2\pi}\int_{-\pi/T}^{\pi/T}\rho(\xi)\re^{\ri\xi x+\lambda_c(\xi)t}(\ri\xi)^2\left(\frac{\Phi_\xi(x)-\phi'(x)}{\ri\xi}\right)\left<\widetilde{\Phi}_\xi,\mathcal{B}(fg)(\xi,\cdot)\right>_{L^2(0,T)}\de \xi\\
&\qquad\qquad
-\frac{\chi(t)}{2\pi}\int_{-\pi/T}^{\pi/T}\rho(\xi)\re^{\ri\xi x+\lambda_c(\xi)t}\ri\xi\left(\frac{\Phi_\xi(x)-\phi'(x)}{\ri\xi}\right)\left<\partial_x\widetilde{\Phi}_\xi,\mathcal{B}(fg)(\xi,\cdot)\right>_{L^2(0,T)}\de \xi.
\end{align*}
\end{proposition}
\begin{proof}
 The proofs of the above identities are essentially the same, so we will just prove the one for $s_p(t)$. Applying~\eqref{e:partialx} and integrating by parts gives
\begin{align*}
\left<\partial_x\widetilde\Phi_\xi,\mathcal{B}\left(f\cdot g\right)(\xi,\cdot)\right>_{L^2(0,T)}
&=-\left<\widetilde\Phi_\xi,\mathcal{B}\left(\partial_xf\cdot g+f\cdot\partial_x g\right)(\xi,\cdot)\right>_{L^2(0,T)}\\
&\qquad	+\ri\xi\left<\widetilde\Phi_\xi,\mathcal{B}\left(f\cdot g\right)(\xi,\cdot)\right>_{L^2(0,T)}.
\end{align*}
Multiplying this equality by $\rho(\xi)\re^{\ri\xi x+\lambda_c(\xi)t}$, integrating and rearranging terms yields the desired result.
\end{proof}

By combining the above identities with our previous linear estimates we obtain the following result.

\begin{lemma}\label{L:mod_bd2}
There exists a constant $C > 0$, and for all integers $\ell,j\geq 0$ there exists a constant $C_{\ell,j}>0$ such that the following inequalities hold:
\begin{align*}
  \left\|\partial_x^\ell\partial_t^j s_p(t)\left(f\cdot\partial_xg\right)\right\|_{L^2}
  & \leq C_{\ell,j}(1+t)^{-\frac{1}{4}-\frac{\ell+j}{2}}\left(\|fg\|_{L^1} + \|\partial_xf\cdot g\|_{L^1}\right),
\\
\left\|\widetilde{S}_c(t)\left(f\cdot\partial_x g\right)\right\|_{L^2}
& \leq C(1+t)^{-\frac{3}{4}}\left(\|fg\|_{L^1} + \|\partial_xf\cdot g\|_{L^1}\right),
\end{align*}
for all $f,g\in H^1(\RM)$ and $t \geq 0$, and
\begin{align*}
  \left\|\partial_x^\ell\partial_t^j s_p(t)\left(f\cdot\partial_x^2g\right)\right\|_{L^2}
  & \leq C_{\ell,j}(1+t)^{-\frac{1}{4}-\frac{\ell+j}{2}}\left(\|fg\|_{L^1} + \|\partial_xf\cdot g\|_{L^1}+\|\partial_x^2f\cdot g\|_{L^1}\right),
\\
\left\|\widetilde{S}_c(t)\left(f\cdot\partial_x^2 g\right)\right\|_{L^2}
& \leq C(1+t)^{-\frac{3}{4}}\left(\|fg\|_{L^1} + \|\partial_xf\cdot g\|_{L^1}+\|\partial_x^2f\cdot g\|_{L^1}\right),
\end{align*}
for all $f,g\in H^2(\RM)$ and $t \geq 0$.
\end{lemma}

\begin{proof}
The first two inequalities follow directly from the identities in Proposition~\ref{P:parts} and the estimates in Lemma~\ref{L:mod_bd}. For the latter ones, we
use
\begin{align*}
\partial_x^2\mathcal{B}(f)(\xi,x)&=\mathcal{B}\left(\partial_x^2f\right)(\xi,x)-2\ri\xi\partial_x\mathcal{B}(f)(\xi,x)+\xi^2\check{f}(\xi,x)\\
&=\mathcal{B}\left(\partial_x^2f\right)(\xi,x)-2\ri\xi\mathcal{B}(\partial_x f)(\xi,x)-\xi^2\check{f}(\xi,x)
\end{align*}
to derive second-order analogues of the identities in Proposition~\ref{P:parts}.  For example, we obtain
\begin{align*}
\frac{\chi(t)}{2\pi}\int_{-\pi/T}^{\pi/T}\rho(\xi)\re^{\ri\xi x+\lambda_c(\xi)t}&\left<\partial_x^2\widetilde\Phi_\xi,\mathcal{B}(fg)(\xi,\cdot)\right>_{L^2(0,T)}\de\xi=
	s_p(t)\left(\partial_x^2f\cdot g+2\partial_xf\cdot\partial_xg+f\cdot\partial_x^2g\right)\\
&\qquad-2\partial_xs_p(t)\left(\partial_x f\cdot g+f\cdot\partial_x g\right) +\partial_x^2s_p(t)\left(fg\right).
\end{align*}
Then using the estimates in Lemma~\ref{L:mod_bd}, as well as Proposition~\ref{P:parts} to  eliminate first order derivatives on $g$, yields the estimate on $s_p(t)\left(f\cdot\partial_x^2g\right)$ and its derivatives. The estimate on $\widetilde S_c$ is obtained in the same way.
\end{proof}

\subsection{Linear stability result} \label{sec:linstab}

For the sake of completeness we state the linear stability result against localized perturbations established in~\cite{HJP20}, which can be readily obtained by combining the estimates in Lemmas~\ref{L:unmod_bd} and~\ref{L:mod_bd} with the semigroup decomposition in~\eqref{e:semidecomp1} and~\eqref{e:semidecomp2}.

\begin{theorem}[Localized Linear Stability, \cite{HJP20}]\label{t:Loc_LinStab} Let $T > 0$ and suppose $\phi$ is a smooth $T$-periodic steady solution of~\eqref{e:LLE} that is diffusively spectrally stable in the sense of Definition~\ref{Def:spec_stab}. Then, there exists a constant $C>0$ such that for any $f\in L^1(\RM)\cap L^2(\RM)$ we have
\[
\left\|\re^{A[\phi]t}f\right\|_{L^2}\leq C(1+t)^{-\frac{1}{4}}\|f\|_{L^1\cap L^2},
\]
for all $t\geq0$. Furthermore, for each $f\in L^1(\RM)\cap L^2(\RM)$ the function $\gamma(x,t) = \left(s_p(t) f\right)(x)$ is smooth in both of its variables and enjoys the estimates
\[
\|\gamma(,t)\|_{L^2}\leq C(1+t)^{-\frac{1}{4}}\|f\|_{L^1\cap L^2}, \qquad \left\|\re^{A[\phi]t}f-\phi'\gamma(t)\right\|_{L^2}\leq C(1+t)^{-\frac{3}{4}}\|f\|_{L^1\cap L^2},
\]
for all $t\geq0$.
\end{theorem}

Theorem~\ref{t:Loc_LinStab} confirms on the linear level that if $\psi$ is a solution of~\eqref{e:LLE} which is initially close in $L^1(\RM)\cap L^2(\RM)$ to $\phi$, then for large time $\psi$ should behave approximately like~\eqref{e:loc_mod_approx}, i.e., $\psi$ should asymptotically behave like a spatio-temporal phase modulation of the underlying periodic wave $\phi$.

\section{Nonlinear Iteration Scheme} \label{sec:nonl_it}

The goal of this section is to introduce the nonlinear iteration scheme that will be employed in~\S\ref{sec:nonlinearstab} to prove our nonlinear stability result, Theorem~\ref{t:Loc_NonLinStab}.
Thus, let $\phi$ be a smooth $T$-periodic steady wave solution of the LLE~\eqref{e:LLE}, which is diffusively spectrally stable, and consider the perturbed solution $\psi(t)$ of~\eqref{e:LLE} with initial condition $\psi(0) = \phi + v_0$, where $v_0 \in L^1(\R) \cap H^4(\R)$ is sufficiently small.

In~Section~\ref{sec:unmod_pert} we study the nonlinear dynamics of the perturbation $\vt(t) = \psi(t) - \phi$, and conclude that the associated linear and nonlinear estimates are too weak to close a nonlinear iteration scheme. Hence, to account for the most critical behavior (which originates from translational invariance of the steady wave $\phi$), we introduce in~Section~\ref{sec:mod_pert} a spatio-temporal phase modulation that tracks the shift of the perturbed solution in space relative to $\phi$. We establish a nonlinear iteration scheme for the modulated perturbation and the phase modulation itself. However, this scheme does not provide control over spatial derivatives of the modulated perturbation, i.e., it exhibits a loss of derivatives. We address this loss of derivatives in~Section~\ref{sec:comp_der} using integration by parts and by appending equations for the unmodulated perturbation to the scheme.

\subsection{The Unmodulated Perturbation} \label{sec:unmod_pert}

Setting
\begin{align*} \vt(t) := \psi(t) - \phi, \end{align*}\
the unmodulated perturbation $\vt$ satisfies
\begin{align}
\left(\partial_t-\mathcal{A}[\phi]\right)\vt = \widetilde{\mathcal{N}}(\vt), \label{e:pert_un}
\end{align}
where $\mathcal{A}[\phi]$ is the linear operator defined by~\eqref{e:Aphi} and the nonlinearity $\widetilde{\mathcal{N}}$ is given by
\begin{align*}
\widetilde{\mathcal{N}}(\vt) :=\mathcal{J}\left[\left(\begin{array}{cc} 3\widetilde v_r^2+\widetilde v_i^2 & 2\widetilde v_r \widetilde v_i\\ 2\widetilde v_r \widetilde v_i & \widetilde v_r^2+3\widetilde v_i^2\end{array}\right)\phi+|\widetilde v|^2\widetilde v\right].
\end{align*}
Using the embedding $H^1(\R) \hookrightarrow L^\infty(\R)$ it is straightforward to check the following estimates on the nonlinearity $\widetilde{\mathcal{N}}$.

\begin{lemma} \label{lem:un_nonl}
For any constant $C > 0$, the inequalities
\begin{align}
\left\|\widetilde{\mathcal N}(\vt)\right\|_{L^1} &\lesssim \left\|\vt\right\|_{L^2}^2, \label{e:un_nonlL1}\\
\left\|\widetilde{\mathcal N}(\vt)\right\|_{H^4} &\lesssim \left\|\vt\right\|_{H^3} \left\|\vt\right\|_{H^2} + \left\|\vt\right\|_{H^4} \left\|\vt\right\|_{H^1}, \nonumber
\end{align}
hold for all $\vt \in H^4(\R)$ with $\|\vt\|_{H^2} \leq C$.
\end{lemma}

The local existence and uniqueness of the perturbation $\vt(t)$ as a solution to~\eqref{e:pert_un} is an immediate consequence of the existence of the semigroup $\re^{\mathcal{A}[\phi]t}$ acting on $H^2(\R)$, the estimates above on the nonlinearity $\widetilde{\mathcal{N}}$, and  classical local existence theory for semilinear evolution problems; see, for instance,~\cite[Proposition~4.3.9]{CA98} and~\cite[Theorem~6.1.3]{Pazy}.

\begin{proposition}[Local Theory for the Unmodulated Perturbation] \label{p:local_unmod} For any $v_0 \in  H^4(\R)$, there exists a maximal time $T_{\max} \in (0,\infty]$ such that~\eqref{e:pert_un} admits a unique solution
\begin{align}\vt \in C\big([0,T_{\max}),H^4(\R)\big) \cap C^1\big([0,T_{\max}),H^2(\R)\big), \label{e:regv}\end{align}
with initial condition $\vt(0) =v_0$. In addition, if $T_{\max} < \infty$, then
\begin{align} \lim_{t \uparrow T_{\max}} \left\|\vt(t)\right\|_{H^2} = \infty. \label{e:vtblowup}\end{align}
\end{proposition}

Ideally, one would hope to control the perturbation $\vt(t)$ over time, and prove that~\eqref{e:vtblowup} cannot occur, which implies that $\vt(t)$, and thus $\psi(t)$, are global solutions.
Naively, one might expect this to be accomplished by integrating~\eqref{e:pert_un} and bounding the perturbation $\vt(t)$ iteratively using the Duhamel formulation
\begin{align}
\vt(t)=\re^{\mathcal{A}[\phi]t}v_0+\int_0^t\re^{\mathcal{A}[\phi](t-s)}\widetilde{\mathcal{N}}(\vt(s))\de s, \label{e:intvt}
\end{align}
for $t \in [0,T_{\max})$. However, as outlined in Remark~\ref{rem:fail} below the temporal bounds on the semigroup $\re^{\mathcal{A}[\phi]t}$,
established in Section~\ref{sec:linear_est}, are too weak to close the resulting nonlinear iteration scheme. Thus, in order to obtain faster linear decay rates, we introduce a spatio-temporal phase modulation $\gamma(x,t)$ for the perturbed solution $\psi$ in the next subsection, which accounts for the most critical behavior of the linear solution operator.

\begin{remark} \label{rem:fail} \upshape
The estimates on the critical component $S_c(t)$ of the semigroup $\re^{\mathcal{A}[\phi]t}$ in Lemma~\ref{L:mod_bd} indicate that the perturbation $\vt(t)$ decays in $H^1(\R)$ at most with rate $(1+t)^{-1/4}$. Thus, in an attempt to close a nonlinear iteration scheme, it makes sense to take $t > 0$ and assume that we have indeed $\|\vt(s)\|_{H^1} \lesssim (1+s)^{-1/4}$ for $s \in [0,t)$. For the next iteration, we then need to show that the right-hand side of~\eqref{e:intvt} decays at least with rate $(1+t)^{-1/4}$. However, Lemma~\ref{L:mod_bd} and estimate~\eqref{e:un_nonlL1} are insufficient to bound the contribution
$$\int_0^t \phi' s_p(t-s) \widetilde{\mathcal N}(\vt(s)) \de s,$$
occurring on the right-hand side of~\eqref{e:intvt}, where we recall the decomposition~\eqref{e:semidecomp2} of critical component $S_c(t)$ of the semigroup. Indeed, one finds the latter to be bounded by
\begin{align*} \int_0^t (1+t-s)^{-\frac{1}{4}}(1+s)^{-\frac{1}{2}} \de s \lesssim (1+t)^{\frac{1}{4}}. \end{align*}
We conclude that the temporal bounds on the critical component $S_c(t)$ of the semigroup $\re^{\mathcal{A}[\phi]t}$, established in Lemmas~\ref{L:mod_bd} are too weak to close a nonlinear iteration scheme.
This is no surprise as the same bounds on the semigroup and nonlinearity can be obtained for the nonlinear heat equation $u_t = u_{xx} + u^2$ in which all
nonnegative, nontrivial initial data in $H^1(\R)$  blow up in finite time~\cite{FU66}.
\end{remark}

\subsection{The Modulated Perturbation} \label{sec:mod_pert}

We now introduce the modulated perturbation by taking
\begin{align}
v(x,t) = \psi(x-\gamma(x,t),t) - \phi(x),
\label{e:modulated}\end{align}
in which the spatio-temporal phase modulation $\gamma(x,t)$ satisfies $\gamma(\cdot,0)=0$, i.e., it vanishes identically at $t=0$. Substituting~\eqref{e:modulated} into the LLE~\eqref{e:LLE}, we obtain the equation
\begin{align}
\left(\partial_t-\mathcal{A}[\phi]\right)\left(v+\gamma\phi'\right)=\mathcal{N}\left(v,\gamma,\partial_t \gamma\right)
	+\left(\partial_t-\mathcal{A}[\phi]\right)\left(\gamma_x v\right), \label{e:pert_mod}
\end{align}
where
\begin{align}
\mathcal{N}(v,\gamma,\gamma_t)=\mathcal{Q}(v,\gamma)+\partial_x \mathcal{R}(v,\gamma,\gamma_t), \label{e:decompN}
\end{align}
with
\[
\mathcal{Q}(v,\gamma)=\left(1-\gamma_x\right)\mathcal{J}\left[\left(\begin{array}{cc} 3v_r^2+v_i^2 & 2v_rv_i\\ 2v_r v_i & v_r^2+3v_i^2\end{array}\right)\phi+|v|^2v\right],
\]
and
\begin{align*}
\mathcal{R}(v,\gamma,\gamma_t)=-\gamma_tv-\beta \mathcal{J}\left[\gamma_{xx}v+2\gamma_xv_x+\frac{\gamma_x^2}{1-\gamma_x}\left(\phi'+v_x\right)\right].
\end{align*}
Using the embedding $H^1(\R) \hookrightarrow L^\infty(\R)$ it is straightforward to check the following estimate on the nonlinearity $\mathcal N$ in~\eqref{e:pert_mod}.

\begin{lemma} \label{lem:mod_nonlL2}
Fix a constant $C > 0$. The inequality
\begin{align*}
\left\|\mathcal{N}(v,\gamma,\gamma_t)\right\|_{L^2} &\lesssim \left\|v\right\|_{L^2} \left\|v\right\|_{H^1} + \left\|(\gamma_x,\gamma_t)\right\|_{H^2 \times H^1} \left(\left\|v\right\|_{H^2} + \left\|\gamma_x\right\|_{L^2}\right),
\end{align*}
holds for $v \in H^2(\R)$ and $(\gamma,\gamma_t) \in H^3(\R) \times H^1(\R)$ satisfying $\|v\|_{H^1} \leq C$ and $\|\gamma\|_{H^2} \leq \frac{1}{2}$.
\end{lemma}

Integrating~\eqref{e:pert_mod} yields the Duhamel formulation
\begin{align}
v(t)+\gamma(t)\phi' = \re^{\mathcal{A}[\phi]t}v_0 + \int_0^t \re^{\mathcal{A}[\phi](t-s)}\mathcal{N}(v(s),\gamma(s),\partial_t \gamma(s))\de s + \gamma_x(t)v(t), \label{e:intv}
\end{align}
where we used the property that both $\gamma(\cdot,0)$ and $\gamma_x(\cdot,0)$ are identically zero. We grouped terms that are nonlinear in $v$, $\gamma$ and their derivatives on the right-hand side of~\eqref{e:intv}, whereas the left-hand side contains all contributions that are linear in $v$, $\gamma$ and their derivatives. The key idea is to make a judicious choice for $\gamma(t)$ such that the linear term $\gamma(t)\phi'$ compensates for the most critical nonlinear contributions in~\eqref{e:intv}. For this, we recall from~\eqref{e:semidecomp1} and~\eqref{e:semidecomp2} that the semigroup $\re^{\mathcal{A}[\phi]t}$ can be decomposed as
\begin{align}
\re^{\mathcal{A}[\phi]t}= \phi' s_p(t) +\widetilde{S}(t), \label{e:semdecomp}
\end{align}
with
\[\widetilde{S}(t) := \widetilde{S}_c(t) + S_e(t).\]
By Lemmas~\ref{L:unmod_bd} and~\ref{L:mod_bd} the slowest temporal decay in~\eqref{e:semdecomp} is exhibited by $\phi' s_p(t)$. This recommends the (implicit) choice
\begin{align}
\gamma(t) = s_p(t)v_0 + \int_0^t s_p(t-s)\mathcal{N}(v(s),\gamma(s),\partial_t \gamma(s))\de s. \label{e:intgamma}
\end{align}
We use this equality as a definition for $\gamma$. Noting that the modulated perturbation $v$ can be written in terms of the unmodulated perturbation $\vt$ as
\begin{align}
v(x,t) = \vt(x-\gamma(x,t),t) + \phi(x-\gamma(x,t)) - \phi(x). \label{e:modulated2}
\end{align}
the equality~\eqref{e:intgamma} (implicitly) defines $\gamma$ as a function of the unmodulated perturbation $\vt$. The existence and uniqueness of a local solution $\gamma$, for a given $\vt$, is established in the following result.

\begin{proposition}[Local Theory for the Phase Modulation] \label{p:gamma}
For $\vt$ and $T_{\max}$ as in Proposition~\ref{p:local_unmod}, there exists a maximal time $\tau_{\max} \in (0,T_{\max}]$ such that~\eqref{e:intgamma} with $v$ given by~\eqref{e:modulated2} has a unique solution
\begin{align*}\gamma \in C\big([0,\tau_{\max}),H^4(\R)\big) \cap C^1\big([0,\tau_{\max}),H^2(\R)\big), \end{align*}
with $\gamma(0)=0$. In addition, if $\tau_{\max} < T_{\max}$, then
\begin{align*} \lim_{t \uparrow \tau_{\max}} \left\|\left(\gamma(t),\partial_t \gamma(t)\right)\right\|_{H^4 \times H^2} = \infty.\end{align*}
\end{proposition}

We prove this proposition  in  Appendix~\ref{app:local}.
Given now the phase modulation $\gamma(t)$ in Proposition~\ref{p:gamma} and the unmodulated perturbation $\vt(t)$  in Proposition~\ref{p:local_unmod}, the modulated perturbation $v(t)$ is uniquely determined by~\eqref{e:modulated2}. More precisely, we have the following result.

\begin{corollary}[The Modulated Perturbation] \label{C:local_v}
For $\vt$ as in Proposition~\ref{p:local_unmod} and $\gamma$ and $\tau_{\max}$ given by Proposition~\ref{p:gamma}, the modulated perturbation $v$ defined by~\eqref{e:modulated2} satisfies $v \in C\big([0,\tau_{\max}),H^2(\R)\big)$. Moreover, the Duhamel formulation~\eqref{e:intv} holds for $t \in [0,\tau_{\max})$.
\end{corollary}

Subtracting~\eqref{e:intgamma} from~\eqref{e:intv} we obtain the equation for the modulated perturbation,
\begin{align}
v(t)=\widetilde{S}(t)v_0+\int_0^t\widetilde{S}(t-s)\mathcal{N}(v(s),\gamma(s),\partial_t \gamma(s)) \de s+\gamma_x(t)v(t), \label{e:intv2}
\end{align}
which holds for $t \in [0,\tau_{\max})$. Notice that those terms exhibiting the slowest temporal decay in~\eqref{e:intv} are canceled out in~\eqref{e:intv2} by our choice of $\gamma(t)$. Indeed, by Lemmas~\ref{L:unmod_bd} and~\ref{L:mod_bd}, the component $\widetilde{S}(t)$ exhibits decay at rate $(1+t)^{-3/4}$, which is faster than the (diffusive) decay at rate $(1+t)^{-1/4}$ of the full semigroup $\re^{\mathcal A[\phi] t}$. Moreover, the nonlinear residual $\mathcal{N}$ depends on derivatives of $\gamma$ only, which, exploiting that $s_p(0) = 0$, satisfy
\begin{align}
\partial_x^\ell\partial_t^j \gamma(t)=\partial_x^\ell\partial_t^js_p(t)v_0+\int_0^t \partial_x^\ell\partial_t^js_p(t-s) \mathcal{N}(v(s),\gamma(s),\partial_t \gamma(s))\de s, \label{e:intgamma2}
\end{align}
for $\ell,j\in\mathbb{N}_0$ and $t \in [0,\tau_{\max})$. By Lemma~\ref{L:mod_bd} the operators $\partial_x^\ell\partial_t^j s_p(t)$ also exhibit fast decay at rate $(1+t)^{-3/4}$ for $\ell + j \geq 1$. Therefore, one could try to close a nonlinear iteration scheme consisting of~\eqref{e:intv2} and~\eqref{e:intgamma2}. This requires control over the spatial derivatives of $v$ appearing in the nonlinearity $\mathcal{N}$, which we establish in the upcoming subsection.

\begin{remark}\label{rem:cut-off}
{ \upshape
The above analysis stresses the importance of the temporal cut-off function $\chi(t)$ in the decomposition~\eqref{e:semidecomp2} of the critical, algebraically decaying, part $S_c(t)$ of the semigroup $\re^{\mathcal A[\phi] t}$. Indeed, due to our choice of $\chi(t)$, the function $\gamma(t)$ is identically zero on $[0,1]$ so that the initial conditions of the modulated and unmodulated perturbation are compatible, cf.~\eqref{e:intvt} and~\eqref{e:intv2}. In addition, taking the temporal derivative of~\eqref{e:intgamma} yields the contribution $s_p(0)\mathcal{N}(v(t),\gamma(t),\partial_t \gamma(t))$, which vanishes due to our choice of $\chi(t)$. This is crucial for obtaining sufficient regularity of $\partial_t \gamma(t)$, cf.~Proposition~\ref{p:loc_gamma}.
We emphasize that the introduction of the cut-off function $\chi(t)$ in~\eqref{e:semidecomp2} does not influence the temporal decay rates established in Lemma~\ref{L:mod_bd}. Indeed, the decay rates of $s_p(t)$ and $\widetilde{S}_c(t)$ are determined by their behavior for large $t$ (for which $\chi(t)$ equals $1$).
}\end{remark}

\subsection{Compensating the Loss of Derivatives} \label{sec:comp_der}

Our goal now is to close the nonlinear iteration scheme consisting of~\eqref{e:intv2} and~\eqref{e:intgamma2} by exploiting the fast temporal decay exhibited by the operators $\widetilde{S}(t)$ and $\partial_x^\ell\partial_t^js_p(t)$ for $\ell + j \geq 1$. As discussed in the introduction, the main obstruction to closing the scheme is the lack of control over the spatial derivatives $v_x$ and $v_{xx}$ occurring in the nonlinearity $\mathcal{N}$ in~\eqref{e:intv2} and~\eqref{e:intgamma2}. Naively, one would hope to control $v_x$ and $v_{xx}$ through their respective integral equations. However, simply differentiating~\eqref{e:intv2}, or~\eqref{e:pert_mod}, introduces third and fourth derivatives of $v$ in the nonlinearities, and thus does not resolve the issue. In addition, we were unable to establish a nonlinear damping estimate for the modulated perturbation equation \eqref{e:pert_mod}, which would provide control over higher-order derivatives of the modulated perturbation in terms of its lower-order derivatives, cf.~Remark~\ref{rem:damping2}. Instead, we address this loss of derivatives by following the approach developed in~\cite{RS18}.

The approach in~\cite{RS18} relies on four crucial observations. The first is that no loss of derivatives arises in the semilinear equations~\eqref{e:pert_un} and~\eqref{e:intvt} for the unmodulated perturbation, as $\widetilde{\mathcal{N}}(\vt)$ does not contain any derivatives of $\vt$. The second is that, using the mean value theorem, the derivatives $v_x$ and $v_{xx}$ of the modulated perturbation can be bounded in terms of the phase modulation $\gamma$, the unmodulated perturbation $\vt$ and their derivatives. Hence, by \emph{appending} the equation~\eqref{e:intvt} for the unmodulated perturbation $\vt$ to the nonlinear iteration scheme we can establish estimates on $\vt$ and its derivatives, and thus on $v_x$ and $v_{xx}$, without loosing derivatives of $v$ or $\vt$. However, the most critical behavior of the semigroup $\re^{\mathcal{A}[\phi]t}$ is not factored out in~\eqref{e:intvt}. Consequently, the estimates on $v_x$ and $v_{xx}$ will be tame. Therefore, it is important to avoid derivatives of $v$ at points where the nonlinearity is paired with the slowest decaying parts of the semigroup. Here, the third and fourth observation come into play: all spatial derivatives of $v$ in the nonlinearity $\mathcal{N}$ are paired with a spatial or temporal derivative of $\gamma$ and the slowest, algebraically decaying parts $s_p(t)$ and $\widetilde{S}_c(t)$ of the semigroup $\re^{\mathcal A[\phi]t}$ are smoothing, cf.~Lemma~\ref{L:mod_bd}. Therefore, whenever possible, we use the integration by parts identities established in Section~\ref{sec:parts} to move derivatives off $v$ onto $\gamma$, $s_p(t)$ or $\widetilde{S}_c(t)$.

\subsubsection{Integration by Parts}

We integrate by parts to get rid of spatial derivatives of $v$ in the algebraically decaying contributions
\begin{align} \label{e:algcont}
\widetilde{S}_c(t-s)\mathcal{N}(v(s),\gamma(s),\partial_t \gamma(s))~~{\rm and}~~\partial_x^\ell\partial_t^j s_p(t-s)\mathcal{N}(v(s),\gamma(s),\partial_t \gamma(s))
\end{align}
in~\eqref{e:intv2} and~\eqref{e:intgamma2}, respectively.  To this end, we decompose $\mathcal{N}$ as in~\eqref{e:decompN}, where $\mathcal{Q}$ contains no derivatives of $v$, and where $\partial_x \mathcal{R}$ is linear in $v$, $v_x$ and $v_{xx}$ and can be written as
\begin{align*}
\partial_x \mathcal{R}(v,\gamma,\gamma_t) =\mathcal{R}_1(\gamma,\gamma_t) v_{xx} + \mathcal{R}_2(\gamma,\gamma_t) v_x + \mathcal{R}_3(\gamma,\gamma_t)v + \mathcal{R}_4(\gamma).
\end{align*}
Using the integration by parts formulas in Section~\ref{sec:parts}, we establish the following  estimates.

\begin{lemma} \label{l:mod_RQnonl}
  Fix a constant $C > 0$. For all integers  $\ell,j$ with $0\leq\ell,j \leq 4$, the inequalities
\begin{align}
\label{e:mod_RQnonl}
\begin{split}
\left\|\mathcal{Q}(v,\gamma)\right\|_{L^1} &\lesssim \left\|v\right\|_{L^2}^2,\\
\left\|\partial_x^\ell\partial_t^j s_p(t)\left(\partial_x \mathcal{R}(v,\gamma,\gamma_t)\right)\right\|_{L^2} &\lesssim (1+t)^{-\frac{1}{4}-\frac{\ell+j}{2}} \left\|(\gamma_x,\gamma_t)\right\|_{H^2 \times H^1}
\left(\|v\|_{L^2} + \left\|\gamma_x\right\|_{L^2} \right), \\
\left\|\widetilde{S}_c(t)\left(\partial_x \mathcal{R}(v,\gamma,\gamma_t)\right)\right\|_{L^2} &\lesssim (1+t)^{-\frac{3}{4}}\left\|(\gamma_x,\gamma_t)\right\|_{H^2 \times H^1} \left(\|v\|_{L^2} + \left\|\gamma_x\right\|_{L^2} \right),
\end{split}
\end{align}
hold for $t \geq 0$, $v \in H^2(\R)$ and $(\gamma,\gamma_t) \in H^3(\R) \times H^1(\R)$ satisfying $\|v\|_{H^1} \leq C$ and $\|\gamma\|_{H^3} \leq \frac{1}{2}$.
\end{lemma}
\begin{proof}
The  first inequality in~\eqref{e:mod_RQnonl} is an immediate consequence of the embedding $H^1(\R) \hookrightarrow L^\infty(\R)$. Moreover, the same embedding also yields
\begin{align*}
\left\|\partial_x^\ell \mathcal{R}_k(\gamma,\gamma_t)\right\|_{L^2}, \left\|\mathcal{R}_4(\gamma)\right\|_{L^2} &\lesssim \left\|(\gamma_x,\gamma_t)\right\|_{H^2 \times H^1},
\end{align*}
for $k = 1,2,3$, nonnegative integers $\ell$ with $k+\ell\leq3$, and any $(\gamma,\gamma_t) \in H^3(\R) \times H^1(\R)$ satisfying $\|\gamma\|_{H^3} \leq \frac{1}{2}$. Now, the last two inequalities in~\eqref{e:mod_RQnonl} follow by applying the integration by parts formulas in Lemma~\ref{L:mod_bd2}.
\end{proof}

Note that the right-hand side of~\eqref{e:mod_RQnonl} does not depend on any derivative of the modulated perturbation $v$. Thus, we have addressed the loss of derivatives in the algebraically decaying contributions~\eqref{e:algcont}.

\subsubsection{Mean Value Inequalities}

In our forthcoming analysis, we need the following inequalities on the difference between the modulated and unmodulated perturbations.
\begin{lemma}[Mean Value Inequalities] \label{l:meanvalue}
For $\vt$ and $v$ given by Proposition~\ref{p:local_unmod} and Corollary~\ref{C:local_v}, respectively, the inequalities
\begin{align}
\begin{split}\label{e:meanvalue}
\left\|v(t) - \vt(t)\right\|_{L^2} &\leq \left(\|\phi'\|_{L^\infty} + \|\vt(t)\|_{H^2}\right) \|\gamma(t)\|_{L^2},\\
\left\|v_x(t) - \vt_x(t)\right\|_{L^2} &\leq \left(\|\phi'\|_{L^\infty} + \|\vt(t)\|_{H^2}\right)\|\gamma_x(t)\|_{L^2} + \left(\|\phi''\|_{L^\infty} + \|\vt(t)\|_{H^3}\right) \|\gamma(t)\|_{L^2},\\
\left\|v_{xx}(t) - \vt_{xx}(t)\right\|_{L^2} &\leq \left(\|\phi'\|_{L^\infty} + \|\vt(t)\|_{H^2}\right) \|\gamma_{xx}(t)\|_{L^2}  + \left(\|\phi'''\|_{L^\infty} + \|\vt(t)\|_{H^4}\right) \|\gamma(t)\|_{L^2}\\
&\qquad\qquad +\, \left(\|\phi''\|_{L^\infty} + \|\vt(t)\|_{H^3}\right) \|\gamma_x(t)\|_{L^2} \left(2+\|\gamma(t)\|_{H^2}\right),\end{split}
\end{align}
hold for all $t \in [0,\tau_{\max})$.
\end{lemma}
\begin{proof}
Recall that by ~\eqref{e:modulated} we have
\begin{align}
v(x,t) - \vt(x,t) = \psi(x-\gamma(x,t),t) - \psi(x,t), \label{e:diffpert}
\end{align}
for $x \in \R$ and $t \in [0,\tau_{\max})$. By applying the mean value theorem to~\eqref{e:diffpert}  we obtain the inequalities
\begin{align}
\begin{split}
\left|v(x,t) - \vt(x,t)\right| &\leq \|\psi_x(t)\|_{L^\infty} |\gamma(x,t)|,\\
\left|v_x(x,t) - \vt_x(x,t)\right| &\leq \|\psi_x(t)\|_{L^\infty} |\gamma_x(x,t)| + \|\psi_{xx}(t)\|_{L^\infty} |\gamma(x,t)|,\\
\left|v_{xx}(x,t) - \vt_{xx}(x,t)\right| &\leq \|\psi_x(t)\|_{L^\infty} |\gamma_{xx}(x,t)| + 2\|\psi_{xx}(t)\|_{L^\infty} |\gamma_x(x,t)|\\
&\qquad + \,\|\psi_{xx}(t)\|_{L^\infty} |\gamma_x(x,t)|^2+ \|\psi_{xxx}(t)\|_{L^\infty} |\gamma(x,t)|,
\end{split} \label{e:meanvalue_pointwise}
\end{align}
for $x \in \R$ and $t \in [0,\tau_{\max})$. Substituting $\psi(t) = \phi + \vt(t)$ in the above inequalities and using the embedding $H^1(\R) \hookrightarrow L^\infty(\R)$, yields the result.
\end{proof}

These mean value inequalities connecting the unmodulated perturbation $\vt$ to the modulated perturbation $v$ allow us to append the equation~\eqref{e:intvt} for $\vt$ to the integral system consisting of the equations~\eqref{e:intv2} and~\eqref{e:intgamma2} for $v$ and $\gamma$, and obtain a nonlinear iteration scheme in
\[
\partial_x^i v(t),\ \partial_x^j \partial_t \gamma(t),\ \partial_x^\ell \gamma(t), \ \partial_x^k \vt(t), \qquad 0\leq i,j\leq 2, \ 0\leq k,\ell\leq 4.
\]
We show in the next section that this nonlinear iteration scheme closes, which yields the proof of Theorem~\ref{t:Loc_NonLinStab}.

\begin{remark}\label{rem:regularity}
\upshape
The mean value inequalities~\eqref{e:meanvalue_pointwise} provide pointwise approximations of the spatial derivatives of the modulated perturbation $v$ by those of the unmodulated perturbation $\vt$. Bounding the right-hand side of~\eqref{e:meanvalue_pointwise} requires $L^\infty$-estimates on the first, second and third spatial derivatives of the perturbed solution $\psi(t) = \phi + \vt(t)$. Hence, any nonlinear iteration scheme exploiting the mean value inequalities~\eqref{e:meanvalue_pointwise} should provide control over the $L^\infty$-norm of $\vt_x,\vt_{xx}$ and $\vt_{xxx}$. In a Hilbertian framework, as ours, such control is given by the $H^4$-norm of $\vt$, $k = 4$ being the smallest integer for which the embedding $H^k(\R) \hookrightarrow W^{3,\infty}(\R)$ holds. This explains the choice $v_0 \in H^4(\R)$ in Theorem~\ref{t:Loc_NonLinStab}. We expect that it is possible to allow for less regular initial data in Theorem~\ref{t:Loc_NonLinStab}. However, the main purpose of this paper is to introduce a working methodology to establish nonlinear stability of steady $T$-periodic waves for the LLE rather than to obtain optimal regularity with respect to localized perturbations.
\end{remark}

\section{Nonlinear Stability Analysis} \label{sec:nonlinearstab}

In this section, we establish  the proof of our main result, Theorem~\ref{t:Loc_NonLinStab}, by applying the linear estimates obtained in~Section~\ref{sec:linear_est}
to the nonlinear iteration scheme consisting of the equations~\eqref{e:intvt},~\eqref{e:intv2},~\eqref{e:intgamma2} and the inequalities~\eqref{e:meanvalue} relating $\vt$, $v$, $\gamma$.

\begin{proof}[Proof of Theorem~\ref{t:Loc_NonLinStab}]
We close a nonlinear iteration scheme, controlling the unmodulated perturbation $\vt \colon [0,T_{\max}) \to H^4(\R)$, the phase modulation $\gamma \colon [0,\tau_{\max}) \to H^4(\R)$ and the modulated perturbation $v  \colon [0,\tau_{\max}) \to H^2(\R)$, all defined in Section~\ref{sec:nonl_it}. By Propositions~\ref{p:local_unmod} and~\ref{p:gamma} and Corollary~\ref{C:local_v}, the template function $\eta \colon [0,\tau_{\max}) \to \R$ given by\footnote{For the motivation behind the choice of temporal weights in the template function $\eta(t)$, we refer to Remark~\ref{rem:choice_eta} below.}
\begin{align*}
\eta(t) &= \sup_{0\leq s\leq t} \left[(1+s)^{\frac{3}{4}}\left(\|v(s)\|_{L^2} + \left\|\partial_x \gamma(s)\right\|_{H^3} + \left\|\partial_t \gamma(s)\right\|_{H^2}\right) + (1+s)^{\frac{1}{4}}\left(\|\vt(s)\|_{L^2} + \left\|\gamma(s)\right\|_{L^2}\right)\right.\\
&\qquad\phantom{\sup_{0\leq s\leq t}} \left. + \, (1+s)^{\frac{1}{8}}\left(\|\vt_x(s)\|_{L^2} + \left\|v_x\right\|_{L^2}\right) +
\|\vt_{xx}(s)\|_{L^2} + \left\|v_{xx}(s)\right\|_{L^2}\right.\\
&\qquad\phantom{\sup_{0\leq s\leq t}} \left. + \, (1+s)^{-\frac{1}{8}}\|\vt_{xxx}(s)\|_{L^2} + (1+s)^{-\frac{1}{4}}\|\vt_{xxxx}(s)\|_{L^2}\right]
\end{align*}
is continuous, positive and monotonically increasing. Moreover, if $\tau_{\max} < \infty$, then it holds
\begin{align}
 \lim_{t \uparrow \tau_{\max}} \eta(t) = \infty. \label{e:blowupeta}
\end{align}
Our approach to closing the nonlinear iteration scheme is to prove that there exist constants $B > 0$ and $C > 1$ such that for all $t \in [0,\tau_{\max})$ with $\eta(t) \leq B$ we have
  \begin{align}
    \eta(t) \leq C\left(E_0 + \eta(t)^2\right), \label{e:etaest}
  \end{align}
with $E_0$ defined in Theorem~\ref{t:Loc_NonLinStab}. Then, provided that $E_0 < \min\{\frac{1}{4C^2},\frac{B}{2C}\}$, it follows $\eta(t) \leq 2CE_0 \leq B$, for all $t \in [0,\tau_{\max})$, by applying continuous induction. Indeed, given that $\eta(s) \leq 2CE_0$ for each $s \in [0,t)$, it follows $\eta(t) \leq C\left(E_0 + 4C^2E_0^2\right) < 2CE_0$ by estimate~\eqref{e:etaest} and continuity of $\eta$. All in all, if~\eqref{e:etaest} holds, then we have $\eta(t) \leq 2CE_0$, for all $t \in [0,\tau_{\max})$, which shows that~\eqref{e:blowupeta} cannot occur. Consequently, it holds $\tau_{\max} = \infty$ and $\eta(t) \leq 2CE_0$ for all $t \geq 0$. Upon taking $\varepsilon = \min\{\frac{1}{4C^2},\frac{B}{2C}\} > 0$ and $M = 2C$ this yields the desired result.

It remains  to prove the key estimate~\eqref{e:etaest}. To this end, take $B = \frac{1}{2}$ and assume $t \in [0,\tau_{\max})$ is such that $\eta(t) \leq B$.
We begin by bounding the modulated perturbation $v(t)$ and the phase modulation $\gamma(t)$ via the integral equations \eqref{e:intv2} and~\eqref{e:intgamma2}, respectively.  Recalling from~\eqref{e:semdecomp} that $\widetilde{S}(t)=\widetilde{S}_c(t)+S_e(t)$, we control the contributions from the operators $\widetilde{S}_c(t)$ and $S_e(t)$ in~\eqref{e:intv2} separately.  To account for the $S_e(t)$-contribution in the convolution term of~\eqref{e:intv2}, note that Lemma~\ref{lem:mod_nonlL2} implies that
\begin{align*}
\left\|\mathcal{N}\left(v(s),\gamma(s),\partial_t \gamma(s)\right)\right\|_{L^2} &\lesssim \eta(s)^2 (1+s)^{-\frac{3}{4}},
\end{align*}
for $s \in [0,t]$, where we use $\eta(t) \leq B$.  Hence, applying Lemma~\ref{L:unmod_bd} we arrive at
\begin{align} \label{e:nlest1}
\left\|\int_0^t S_e(t-s) \mathcal{N}\left(v(s),\gamma(s),\partial_t \gamma(s)\right) \de s\right\|_{L^2} \lesssim \int_0^t \frac{\eta(s)^2 \re^{-\mu(t-s)}}{(1+s)^{\frac{3}{4}}} \de s \lesssim \frac{\eta(t)^2}{\left(1+t\right)^{\frac{3}{4}}}.
\end{align}
To control the remaining terms in~\eqref{e:intv2}-\eqref{e:intgamma2} we note that Lemma~\ref{l:mod_RQnonl} implies
\begin{align*}
\left\|\mathcal{Q}\left(v(s),\gamma(s)\right)\right\|_{L^1} &\lesssim \eta(s)^2 (1+s)^{-\frac{3}{2}},\\
\left\|\partial_x^\ell\partial_t^j s_p(t-s)\left(\partial_x \mathcal{R}\left(v(s),\gamma(s),\partial_t \gamma(s)\right)\right)\right\|_{L^2} &\lesssim \eta(s)^2(1+t-s)^{-\frac{1}{4}-\frac{\ell+j}{2}} (1+s)^{-\frac{3}{2}},\\
\left\|\widetilde{S}_c(t-s)\left(\partial_x \mathcal{R}\left(v(s),\gamma(s),\partial_t \gamma(s)\right)\right)\right\|_{L^2} &\lesssim \eta(s)^2(1+t-s)^{-\frac{3}{4}}(1+s)^{-\frac{3}{2}},
\end{align*}
for $s \in [0,t]$ and $\ell,j \in \mathbb{N}_0$ with $\ell,j \leq 4$, where we use $\eta(t) \leq B$. So, applying Lemma~\ref{L:mod_bd} and recalling~\eqref{e:decompN} we establish that
\begin{align}
\begin{split}
\left\|\int_0^t s_p(t-s) \mathcal{N}\left(v(s),\gamma(s),\partial_t \gamma(s)\right) \de s\right\|_{L^2} &\lesssim \int_0^t \frac{\eta(s)^2}{(1+t-s)^{\frac{1}{4}}(1+s)^{\frac{3}{2}}} \de s \lesssim \frac{\eta(t)^2}{\left(1+t\right)^{\frac{1}{4}}},\\
\left\|\int_0^t \partial_x^\ell\partial_t^j s_p(t-s) \mathcal{N}\left(v(s),\gamma(s),\partial_t \gamma(s)\right) \de s\right\|_{L^2} &\lesssim \int_0^t \frac{\eta(s)^2}{(1+t-s)^{\frac{3}{4}}(1+s)^{\frac{3}{2}}} \de s \lesssim \frac{\eta(t)^2}{\left(1+t\right)^{\frac{3}{4}}},\\
\left\|\int_0^t \widetilde{S}_c(t-s) \mathcal{N}\left(v(s),\gamma(s),\partial_t \gamma(s)\right) \de s\right\|_{L^2} &\lesssim \int_0^t \frac{\eta(s)^2}{(1+t-s)^{\frac{3}{4}}(1+s)^{\frac{3}{2}}} \de s \lesssim \frac{\eta(t)^2}{\left(1+t\right)^{\frac{3}{4}}},
\end{split} \label{e:nlest2}
\end{align}
for $\ell,j \in \mathbb{N}_0$ with $1 \leq \ell + 2j \leq 4$. Thus, using Lemma~\ref{L:mod_bd}, the decomposition~\eqref{e:semdecomp} of $\widetilde{S}(t)$ and estimates~\eqref{e:nlest1} and~\eqref{e:nlest2}, we bound the right-hand sides of~\eqref{e:intgamma},~\eqref{e:intv2} and~\eqref{e:intgamma2} and obtain
\begin{align} \label{e:nlest3}
\|\gamma(t)\|_{L^2} \lesssim \frac{E_0 + \eta(t)^2}{(1+t)^{\frac{1}{4}}}, \qquad \|v(t)\|_{L^2}, \left\|\partial_x \gamma(t)\right\|_{H^3}, \left\|\partial_t \gamma(t)\right\|_{H^2} \lesssim \frac{E_0 + \eta(t)^2}{(1+t)^{\frac{3}{4}}}.
\end{align}

It remains now to provide control over the $L^2$-norms of $v_x$ and $v_{xx}$.  To this end, we proceed with establishing estimates on the unmodulated perturbation $\vt(t)$ and its derivatives, with the goal
of then using the mean value inequalities in Lemma~\ref{l:meanvalue} to infer control on the derivatives of $v$.
An estimate on the $L^2$-norm of $\vt(t)$ follows readily by the mean value inequalities. Indeed, combining Lemma~\ref{l:meanvalue} with~\eqref{e:nlest3} yields
\begin{align} \label{e:nlest5}
\left\|\vt(t)\right\|_{L^2} &\lesssim \|v(t)\|_{L^2} + \left(\|\phi'\|_{L^{\infty}} + \|\vt(t)\|_{H^2}\right) \|\gamma(t)\|_{L^2} \lesssim \frac{E_0 + \eta(t)^2}{(1+t)^{\frac{1}{4}}},
\end{align}
where we use $\eta(t) \leq B$. Next, we establish a bound on the derivative $\vt_{xxxx}(t)$. To this end, note that  Lemma~\ref{lem:un_nonl} implies that
\begin{align} \label{e:nlest4}
\left\|\widetilde{\mathcal N}(\vt(s))\right\|_{L^1} &\lesssim \eta(s)^2(1+s)^{-\frac{1}{2}},  \qquad \left\|\widetilde{\mathcal N}(\vt(s))\right\|_{H^4} \lesssim \eta(s)^2(1+s)^{\frac{1}{8}},
\end{align}
for $s \in [0,t]$, where we use $\eta(t) \leq B$. Thus, differentiating~\eqref{e:intvt} four times with respect to $x$, and using Lemma~\ref{L:unmod_bd}, the decomposition~\eqref{e:semidecomp1}
of the semigroup $\re^{\mathcal{A}[\phi] t}$, and the estimates~\eqref{e:nlest4} we obtain the bound
\begin{align} \label{e:nlest6}
\begin{split}
\left\|\vt_{xxxx}(t)\right\|_{L^2} &\lesssim \left(\re^{-\mu t} + (1+t)^{-\frac{3}{4}}\right)E_0 + \int_0^t \frac{\eta(s)^2 (1+s)^{\frac{1}{8}}}{\re^{\mu(t-s)}} \de s + \int_0^t \frac{\eta(s)^2}{(1+t-s)^{\frac{1}{4}}(1+s)^{\frac{1}{2}}} \de s\\
&\lesssim \left(E_0 + \eta(t)^2\right)(1+t)^{\frac{1}{4}}.
\end{split}
\end{align}
Using the Gagliardo-Nirenberg inequality
\[
\left\|\partial_x^j \vt(t)\right\|_{L^2}\lesssim \left\|\partial_x^4\vt(t)\right\|_{L^2}^{j/4}\left\|\vt(t)\right\|_{L^2}^{1-j/4},\qquad j=1,2,3,
\]
to interpolate between~\eqref{e:nlest5} and~\eqref{e:nlest6}, we readily arrive at
\begin{align} \label{e:nlest7}
\begin{split}
\left\|\vt_{x}(t)\right\|_{L^2} &\lesssim \frac{E_0 + \eta(t)^2}{(1+t)^{\frac{1}{8}}}, \quad \left\|\vt_{xx}(t)\right\|_{L^2} \lesssim E_0 + \eta(t)^2, \quad \left\|\vt_{xxx}(t)\right\|_{L^2} \lesssim \left(E_0 + \eta(t)^2\right)(1+t)^{\frac{1}{8}}.
\end{split}
\end{align}

Subsequently, we employ the mean value inequalities in Lemma~\ref{l:meanvalue} to bound the derivatives of the modulated perturbation $v(t)$ in terms of
derivatives of the unmodulated perturbation $\vt(t)$. Specifically, combining the bounds~\eqref{e:meanvalue} with the estimates~\eqref{e:nlest3} and~\eqref{e:nlest7}, we obtain
\begin{align} \label{e:nlest8}
\begin{split}
\left\|v_x(t)\right\|_{L^2} &\lesssim \left\|\vt_x(t)\right\|_{L^2} + \left(\|\phi\|_{W^{2,\infty}} + \|\vt(t)\|_{H^3}\right)\|\gamma(t)\|_{H^1} \lesssim \frac{E_0 + \eta(t)^2}{(1+t)^{\frac{1}{8}}},\\
\left\|v_{xx}(t)\right\|_{L^2} &\lesssim \left\|\vt_{xx}(t)\right\|_{L^2} + \left(\|\phi\|_{W^{3,\infty}} + \|\vt(t)\|_{H^4}\right) \|\gamma(t)\|_{H^2} \lesssim E_0 + \eta(t)^2,
\end{split}
\end{align}
where we use $\eta(t) \leq B$.

Finally, by estimates~\eqref{e:nlest3},~\eqref{e:nlest5},~\eqref{e:nlest6},~\eqref{e:nlest7} and~\eqref{e:nlest8} it follows that there exists a constant $C > 1$, which is independent of $E_0$ and $t$, such that the key inequality~\eqref{e:etaest} is satisfied, which, as discussed previously, completes the proof of Theorem~\ref{t:Loc_NonLinStab}.
\end{proof}

\begin{remark}\label{rem:choice_eta}
\upshape
The choice of temporal weights in the template function $\eta(t)$ used in the proof of Theorem~\ref{t:Loc_NonLinStab} can be motivated as follows.
First, the weights applied to the terms in $\eta(t)$ involving $\|v(t)\|_{L^2}$, $\|\vt(t)\|_{L^2}$, $\|\gamma(t)\|_{L^2}$, $\|\partial_t \gamma(t)\|_{H^2}$ and $\|\partial_x \gamma(t)\|_{H^3}$ are given by the linear theory. Indeed, Lemmas~\ref{L:unmod_bd} and~\ref{L:mod_bd} imply that the linear term $\widetilde{S}(t)v_0$ in the integral equation~\eqref{e:intv2} for $v(t)$ exhibits $\left(1+t\right)^{-3/4}$-decay, whereas the linear terms $\re^{\mathcal A[\phi] t} v_0$ and $\partial_x^\ell \partial_t^j s_p(t)$ in the integral equations~\eqref{e:intvt} and~\eqref{e:intgamma2} for $\vt(t)$ and $\gamma(t)$ exhibit decay at rates $(1+t)^{-1/4}$ and $(1+t)^{-1/4 - (\ell+j)/2}$, respectively.
Next, the temporal weight applied to the contribution $\|\vt_{xxxx}(t)\|_{L^2}$ in $\eta(t)$ arises by bounding the most critical nonlinear term in the integral equation~\eqref{e:intvt} for $\vt(t)$, which, as outlined in Remark~\ref{rem:fail}, grows at rate $(1+t)^{1/4}$. Finally, the weights applied to $\|\vt_{x}(t)\|_{L^2}, \|\vt_{xx}(t)\|_{L^2}$ and $\|\vt_{xxx}(t)\|_{L^2}$ arise by interpolation, whereas the weights applied to $\|v_x(t)\|_{L^2}$ and $\|v_{xx}(t)\|_{L^2}$ are directly linked to those applied to $\|\vt_{x}(t)\|_{L^2}$ and $\|\vt_{xx}(t)\|_{L^2}$.
In particular, while the bounds stated in Theorem~\ref{t:Loc_NonLinStab} are sharp, the above proof yields additional $L^2$-bounds on the derivatives of $v$ which are \emph{not} expected to be sharp. Indeed, their estimates rely on tame estimates on the unmodulated perturbation $\vt$.
\end{remark}

\appendix
\renewcommand*{\thesection}{\Alph{section}}

\section{Nonlinear damping estimates for the unmodulated perturbation} \label{sec:alternative}

In this subsection, we establish nonlinear damping estimates of the form~\eqref{e:damping} for the unmodulated perturbation $\vt$ exploiting the fact that it satisfies the semilinear equation~\eqref{e:pert_un}. The damping estimates yield tame bounds on the derivative of $\vt$ and, thus, provide an alternative to the Duhamel-based estimates~\eqref{e:nlest6} and~\eqref{e:nlest7} in the proof of Theorem~\ref{t:Loc_NonLinStab}.

To control the $L^2$-norm of the $j$-th derivative of the unmodulated perturbation $\vt$, it makes sense to look at the energy $E_j(t) = \|\partial_x^j \vt(t)\|_{L^2}^2, j \in \mathbb{N}$. The relevant bilinear terms in $\partial_t E_j(t)$ are   
\begin{align*} \left\langle \mathcal A[\phi] \partial_x^j \vt,\partial_x^j \vt\right\rangle_{L^2} + \left\langle \partial_x^j \vt,\mathcal A[\phi] \partial_x^j \vt\right\rangle_{L^2} = -2E_j(t) + \left\langle M[\phi] \partial_x^j \vt,\partial_x^j \vt\right\rangle_{L^2},\end{align*}
where
\begin{align*}
M[\phi] =  2\left(\begin{array}{cc} -2\phi_r\phi_i & \phi_r^2 -\phi_i^2 \\ \phi_r^2 - \phi_i^2 & 2\phi_r \phi_i\end{array}\right),
\end{align*}
corresponds to the remaining symmetric part of the linear operator $\mathcal A[\phi]$. To remove the residual symmetric term, which is currently obstructing a damping estimate, we introduce the modified energy 
\begin{align*} \widetilde{E}_j(t) = \left\|\partial_x^j \vt(t)\right\|_{L^2}^2 - \frac{1}{2\beta} \left\langle \mathcal{J} M[\phi] \partial_x^{j-1} \vt(t), \partial_x^{j-1} \vt(t)\right\rangle_{L^2}.\end{align*}
We emphasize that $\widetilde{E}_j(t)$ still provides control over the $L^2$-norm of the $j$-th derivative of the unmodulated perturbation. Indeed, using Sobolev interpolation we obtain a constant $K > 0$ such that
\begin{align} \left\|\partial_x^j \vt(t)\right\|_{L^2}^2 \leq 2\widetilde{E}_j(t) + K\left\|\vt(t)\right\|_{L^2}^2. \label{e:NLdamp1}\end{align}
Denoting
\begin{align*} 
B[\phi] := \left(\begin{array}{cc} 3\phi_{r}^2 + \phi_{i}^2 & 2\phi_{r}\phi_{i} \\
   2\phi_{r}\phi_{i} & \phi_{r}^2 + 3\phi_{i}^2\end{array}\right),
\end{align*}
we then find
\begin{align*}
\partial_t \widetilde E_j(t) = -2\widetilde E_j(t) + R_1(t) + R_2(t),
\end{align*}
where $R_1(t)$ contains all irrelevant bilinear terms
\begin{align*}
R_1(t) &= \frac{1}{\beta}\left( \Re\left\langle \mathcal{J}M[\phi]\partial_x^{j-1}\left(I + \mathcal{J}(\alpha - B[\phi]))\vt(t)\right),\partial_x^{j-1} \vt(t)\right\rangle_{L^2}-\left\langle \mathcal{J} M[\phi] \partial_x^{j-1}\vt(t), \partial_x^{j-1} \vt(t)\right\rangle_{L^2}\right)\\ 
& \qquad + \, 2\Re\left(\left\langle \mathcal{J} \left(\partial_x^j \left(B[\phi]\vt(t)\right) - B[\phi]\partial_x^j\vt(t)\right),\partial_x^j \vt(t) \right\rangle_{L^2} -  \left\langle\left(\partial_x M[\phi]\right)\partial_x^{j-1} \vt(t), \partial_x^j \vt(t)\right\rangle_{L^2}\right),
\end{align*}
and $R_2(t)$ is the nonlinear residual
\begin{align*}
R_2(t) = 2\Re\left(\left\langle \partial_x^j \mathcal{N}(\vt(t)),\partial_x^j \vt(t)\right\rangle_{L^2} - \frac{1}{2\beta}\left\langle \mathcal{J}M[\phi]\partial_x^{j-1} \mathcal{N}(\vt(t)),\partial_x^{j-1} \vt(t)\right\rangle_{L^2}\right).
\end{align*}
The irrelevant bilinear terms can be estimated with the aid of Sobolev interpolation and the Cauchy-Schwarz and Young inequalities as
\begin{align*} |R_1(t)| \leq \frac{1}{2}\left\|\partial_x^j \vt(t)\right\|_{L^2}^2 + C_1\left\|\vt(t)\right\|_{L^2}^2,\end{align*}
for some constant $C_1 > 0$. On the other hand, using Sobolev interpolation, the Cauchy-Schwarz and Young inequalities and the embedding $H^1(\R) \hookrightarrow L^\infty(\R)$, we find that the nonlinear residual enjoys the estimate
\begin{align*} |R_2(t)| \leq C_2\|\vt(t)\|_{H^j}\left(\left\|\partial_x^j\vt(t)\right\|_{L^2}^2 + \left\|\vt(t)\right\|_{L^2}^2\right),\end{align*}
for some constant $C_2 > 0$, as long as $\|\vt(t)\|_{H^j}$ is bounded. Hence, assuming $\|\vt(t)\|_{H^j}$ is sufficiently small, we obtain the desired nonlinear damping estimate
\begin{align*}
\partial_t \widetilde E_j(t) = -\widetilde E_j(t) + C\left\|\vt(t)\right\|_{L^2}^2,
\end{align*}
for some constant $C > 0$. Integrating the latter and using~\eqref{e:NLdamp1}, we arrive at
\begin{align}
\left\|\partial_x^j \vt(t)\right\|_{L^2}^2 &\leq 2\re^{-t}\widetilde{E}_j(0) + K\left\|\vt(t)\right\|_{L^2}^2 + 2C\int_0^t \re^{-(t-s)} \left\|\vt(s)\right\|_{L^2}^2 \de s. \label{e:finalNLdamp}
\end{align}
Note that, estimate~\eqref{e:finalNLdamp} for $j = 1,2,3,4$, coupled with \eqref{e:nlest5}, could be used to control the derivatives of the unmodulated perturbation in the proof of Theorem~\ref{t:Loc_NonLinStab}, replacing the estimates~\eqref{e:nlest6} and~\eqref{e:nlest7}. 

\begin{remark}\label{rem:damping2} {\upshape
While the above establishes a nonlinear damping estimate for the \emph{unmodulated} perturbation $\tilde{v}$, we emphasize again that
we were unable to establish such a nonlinear damping estimate for the \emph{modulated} perturbation $v$ by following the same strategy. 
The main reason is that the nonlinear term $\partial_x \mathcal R(v,\gamma,\gamma_t)$ in~\eqref{e:pert_mod} gives rise to terms in $\partial_t E_j(t)$, which are bilinear in $(\partial_x^{j+1} v(t),\partial_x^{j} v(t))$ and in $(\partial_x^{j+1} v(t),\partial_x^{j+1} v(t))$ (after integrating by parts). Although it turns out that the  terms, which are bilinear in $(\partial_x^{j+1} v(t),\partial_x^{j+1} v(t))$ cancel, we have not identified a reason why the same should hold for the terms, which are bilinear in $(\partial_x^{j+1} v(t),\partial_x^{j} v(t))$. More precisely, those terms are given by
\begin{align*}
-2\Re \left\langle \left(\partial_t \gamma(t)+\beta \mathcal{J}\partial_x \left(\frac{\gamma_x(t)}{1-\gamma_x(t)}\right)\right)\partial_x^{j+1} v(t),\partial_x^j v(t)\right\rangle_{L^2},
\end{align*}
and are a priori not controlled by the energy $E_j(t)$.  Whether a nonlinear damping estimate of the form \eqref{e:damping} exists for the modulated perturbation
remains an interesting open question.}
\end{remark}

\section{Local Theory for the Phase Modulation} \label{app:local}

The result in Proposition~\ref{p:gamma} is a consequence of the result in Proposition~\ref{p:loc_gamma} below. First, we prove the following preliminary result.

\begin{lemma} \label{L:loc_gamma}
For $\vt$ given by Proposition~\ref{p:local_unmod}, the mapping $V \colon H^2(\R) \times [0,T_{\max}) \to H^2(\R)$ given by
\begin{align*}V(\gamma,t)[x] &= \vt(x-\gamma(x),t) + \phi(x-\gamma(x)) - \phi(x),\end{align*}
is well-defined, continuous in $t$, and locally Lipschitz continuous in $\gamma$ (uniformly in $t$ on compact subintervals of $[0,T_{\max})$).
\end{lemma}
\begin{proof}
First, we note the embedding $H^4(\R) \hookrightarrow C_{b}^3(\R)$ implies that
\begin{align} \vt \in C\big([0,T_{\max}),C_b^3(\R)\big), \label{e:regv3}\end{align}
where $C_b^3(\R)$ denotes the space of $3$-times differentiable functions, whose derivatives are continuous and bounded. Therefore, the mean value theorem yields
\begin{align}\label{e:estV1}
\begin{split}
\|V(\gamma_1,t) - V(\gamma_2,t)\|_{H^2} &\lesssim \|\vt(t) + \phi\|_{W^{3,\infty}} \|\gamma_1 - \gamma_2\|_{H^2},
\end{split}
\end{align}
for $\gamma_{1,2} \in H^2(\R)$ and $t \in [0,T_{\max})$. Taking $\gamma_2 = 0$ in~\eqref{e:estV1} and noting that $V(0,t) = \vt(t) \in H^2(\R)$, shows that $V$ is well-defined. Moreover,~\eqref{e:regv3} and~\eqref{e:estV1} yield Lipschitz continuity of $V$ in $\gamma$ (uniformly in $t$ on compact subintervals of $[0,T_{\max})$).

Similarly as in~\eqref{e:estV1}, we employ the mean value theorem and~\eqref{e:regv3} to obtain
\begin{align*}
\|V(\gamma,t) - V(\gamma,s)\|_{H^{2}} &\lesssim \left\|\left(V(\gamma,t) - V(\gamma,s)\right) - \left(V(0,t) - V(0,s)\right)\right\|_{H^{2}} + \|V(0,t) - V(0,s)\|_{H^2}\\
 &\lesssim \|\vt(t) - \vt(s)\|_{W^{3,\infty}} \|\gamma\|_{H^2} + \|\vt(t) - \vt(s)\|_{H^2},
\end{align*}
for $\gamma \in H^2(\R)$ and $s,t \in [0,T_{\max})$. Continuity of $V$ with respect to $t$ now follows by~\eqref{e:regv} and~\eqref{e:regv3}.
\end{proof}

\begin{proposition} \label{p:loc_gamma}
For $\vt$ given by Proposition~\ref{p:local_unmod}, let $V \colon H^2(\R) \times [0,T_{\max}) \to H^2(\R)$ be the mapping in Lemma~\ref{L:loc_gamma}. Then, there exists a maximal time $\tau_{\max} \in (0,T_{\max}]$ such that the integral system
\begin{align} \label{e:intsysgamma}
\begin{split}
\gamma(t) &= s_p(t)v_0 + \int_0^t s_p(t-s)\mathcal{N}(V(\gamma(s),s),\gamma(s),\gamma_t(s))\de s,\\
\gamma_t(t) &= \partial_t s_p(t)v_0 + \int_0^t \partial_t s_p(t-s)\mathcal{N}(V(\gamma(s),s),\gamma(s),\gamma_t(s))\de s,
\end{split}
\end{align}
has a unique solution
\begin{align*} (\gamma,\gamma_t) \in C\big([0,\tau_{\max}),H^4(\R) \times H^2(\R)\big).\end{align*}
In addition, if $\tau_{\max} < T_{\max}$, then
\begin{align} \label{e:gammablowup2}
\lim_{t \uparrow \tau_{\max}} \left\|\left(\gamma,\gamma_t\right)\right\|_{H^4 \times H^2} = \infty,
\end{align}
holds. Finally,  $\gamma \in C^1\big([0,\tau_{\max}),H^2(\R)\big)$ and $\partial_t \gamma(t) = \gamma_t(t)$ for $t \in [0,\tau_{\max})$.
\end{proposition}
\begin{proof}
First, the result in Lemma~\ref{L:mod_bd} implies that the operators $s_p(t) \colon L^2(\R) \to H^4(\R)$ and $\partial_t s_p(t) \colon L^2(\R) \to H^2(\R)$ are $t$-uniformly bounded and strongly continuous on $[0,\infty)$.
Next, recall that the nonlinearity $\mathcal{N}$ can be decomposed as in~\eqref{e:decompN}, where $\mathcal{Q}$ contains no derivatives of $v$ and $\partial_x \mathcal{R}$ is linear in $v$, $v_x$ and $v_{xx}$. Then, it follows from Lemmas~\ref{lem:mod_nonlL2} and~\ref{L:loc_gamma} that the nonlinear map $N \colon H^4(\R) \times H^2(\R) \times [0,T_{\max}) \to L^2(\R)$ given by
$$N(\gamma,\gamma_t,t) = \mathcal{N}(V(\gamma,t),\gamma,\gamma_t),$$
is well-defined, continuous in $t$, and locally Lipschitz continuous in $(\gamma,\gamma_t)$ (uniformly in $t$ on compact subintervals of $[0,\tau_{\max})$), where we used the inequalities
$$\left\|\partial_x^\ell f \cdot \partial_x^k g\right\|_{L^2} \leq \|f\|_{H^2}\|g\|_{H^4}, \qquad  0\leq k \leq 3, \ 0\leq l \leq 2,$$
to bound the $L^2$-norm of products for functions $f \in H^2(\R)$ and $g \in  H^4(\R)$.

Standard arguments, see for instance~\cite[Proposition 4.3.3]{CA98} or~\cite[Theorem 6.1.4]{Pazy}, now imply that there exist constants $R > 0$ and $\tau \in (0,T_{\max})$ such that $\Psi \colon C\big([0,\tau],B(R)\big) \to C\big([0,\tau],B(R)\big)$ given by
$$\Psi(\gamma,\gamma_t)[t] = \begin{pmatrix} s_p(t)v_0 \\ \partial_t s_p(t)v_0\end{pmatrix} + \int_0^t \begin{pmatrix} s_p(t-s)\mathcal{N}(V(\gamma(s),s),\gamma(s),\gamma_t(s)) \\
\partial_t s_p(t-s)\mathcal{N}(V(\gamma(s),s),\gamma(s),\gamma_t(s))\end{pmatrix} \de s,$$
is a well-defined contraction mapping, where $B(R)$ is the closed ball centered at the origin in $H^4(\R) \times H^2(\R)$ of radius $R$. Hence, by the Banach fixed point theorem, $\Psi$ admits a unique fixed point, which yields a unique solution $(\gamma,\gamma_t) \in C\big([0,\tau], H^4(\R) \times H^2(\R)\big)$ to~\eqref{e:intsysgamma}.
Letting $\tau_{\max} \in (0,T_{\max}]$ be the supremum of all such $\tau$, we obtain a maximally defined solution $(\gamma,\gamma_t) \in C\big([0,\tau_{\max}), H^4(\R) \times H^2(\R)\big)$ to~\eqref{e:intsysgamma}.

  Next, assume by contradiction that $\tau_{\max} < T_{\max}$ and~\eqref{e:gammablowup2} does not hold. Take $t_0 \in [0,\tau_{\max})$. Similarly as before, one proves that there exist constants $M,\delta > 0$, which are independent of $t_0$, such that $\Psi_{t_0} \colon C\big([t_0,t_0+\delta],B(M)\big) \to C\big([t_0,t_0+\delta],B(M)\big)$ given by
\begin{align*}
\Psi_{t_0}(\widetilde{\gamma},\widetilde{\gamma}_t)[t] &= \begin{pmatrix} s_p(t)v_0 \\ \partial_t s_p(t)v_0\end{pmatrix} + \int_0^{t_0} \begin{pmatrix} s_p(t-s)\mathcal{N}(V(\gamma(s),s),\gamma(s),\gamma_t(s)) \\
\partial_t s_p(t-s)\mathcal{N}(V(\gamma(s),s),\gamma(s),\gamma_t(s))\end{pmatrix} \de s\\
& \qquad + \int_{t_0}^t \begin{pmatrix} s_p(t-s)\mathcal{N}(V(\widetilde{\gamma}(s),s),\widetilde{\gamma}(s),\widetilde{\gamma}_t(s)) \\
\partial_t s_p(t-s)\mathcal{N}(V(\widetilde{\gamma}(s),s),\widetilde{\gamma}(s),\widetilde{\gamma}_t(s))\end{pmatrix} \de s,
\end{align*}
is a well-defined contraction mapping, which admits a unique fixed point $(\widetilde\gamma,\widetilde\gamma_t) \in C\big([t_0,t_0+\delta], H^4(\R) \times H^2(\R)\big)$. Setting $t_0 := \tau_{\max} - \delta/2$, it readily follows that $(\check{\gamma},\check{\gamma}_t) \in C\big([0,\tau_{\max} + \delta/2],H^4(\R) \times H^2(\R)\big)$ given by
\begin{align*}
 (\check{\gamma}(t),\check{\gamma}(t)) = \begin{cases} (\gamma(t),\gamma_t(t)), & t \in [0,\tau_{\max} - \frac{\delta}{2}], \\ (\widetilde\gamma(t),\widetilde\gamma_t(t)), & t \in [\tau_{\max} - \frac{\delta}{2},\tau_{\max} + \frac{\delta}{2}],
                     \end{cases}
\end{align*}
solves~\eqref{e:intsysgamma}, which contradicts the maximality of $\tau_{\max}$. We conclude that if $\tau_{\max} < T_{\max}$, then~\eqref{e:gammablowup2} must hold.

Finally, Lemma~\ref{L:mod_bd} readily implies that $\gamma(t)$ is differentiable on $[0,\tau_{\max})$ with $\partial_t \gamma(t) = \gamma_t(t)$, where we use $s_p(0) = 0$. This completes the proof.
\end{proof}

\bibliographystyle{abbrv}
\bibliography{LLE}

\end{document}